\documentclass[10pt,reqno,final]{amsart}
\usepackage{epsfig,amssymb,amsmath,version}
\usepackage{amssymb,version,graphicx,fancybox,mathrsfs,multirow}
\usepackage{epstopdf}
\usepackage{url,hyperref}
\usepackage[notcite,notref]{showkeys}
\usepackage{bm}
\usepackage{subfigure}
\usepackage{color,xcolor}
\usepackage{cases}
\usepackage{mathtools}
\usepackage{extarrows}
\usepackage{bbm}

\catcode`\@=11 \theoremstyle{plain}
\@addtoreset{equation}{section}   

\@addtoreset{figure}{section}
\renewcommand\thefigure{\thesection.\@arabic\c@figure}
\renewcommand{\thefigure}{\arabic{section}.\arabic{figure}}
\newtheorem{thm}{\bf Theorem}

\newenvironment{theorem}{\begin{thm}} {\end{thm}}
\newtheorem{cor}{\bf Corollary}

\newtheorem{lmm}{\bf Lemma}

\newenvironment{lemma}{\begin{lmm}}{\end{lmm}}

\theoremstyle{example}
\newtheorem{example}{Example}[section]

\theoremstyle{remark}
\newtheorem{rem}{\bf Remark}[section]
\theoremstyle{definition}

\numberwithin{table}{section}

\def \af {\alpha}

\newcommand{\bs}[1]{\boldsymbol{#1}}

\renewcommand \wedge \times

\begin{document}
	\bibliographystyle{plain}
	\graphicspath{{./figures/}}

	\title[Cole-Cole media
	] {An accurate  spectral method for Maxwell equations in Cole-Cole dispersive media}
		\author[
	C. Huang   \; $\&$ \; L. Wang
	]{
		\;\; Can Huang${}^1$ \;\; and\;\; Li-Lian Wang${}^{2}$
		}
	
	\thanks{${}^1$School of Mathematical Sciences  and Fujian Provincial Key Laboratory on Mathematical Modeling \& High Performance Scientific Computing, Xiamen University, Fujian 361005, China. The research of this author is supported by National Natural Science Foundation of China  (No. 11401500, 91630204). \\
		\indent ${}^{2}$Division of Mathematical Sciences, School of Physical
		and Mathematical Sciences, Nanyang Technological University,
		637371, Singapore. The research of this author is partially supported by Singapore MOE AcRF Tier 1 Grant (RG 15/12). \\
		\indent The authors would like to thank both universities for hosting their mutual visits to complete this work. 
			}
	
\begin{abstract}
In this paper, we propose an accurate numerical means built upon a spectral-Galerkin method in spatial discretization and an enriched multi-step spectral-collocation approach in temporal direction,  for Maxwell equations in Cole-Cole dispersive media in two-dimensional setting.  Our starting point is to  derive a new model involving only one unknown field from the original model with three unknown fields: electric, magnetic fields and  the induced electric polarisation (described by a global temporal convolution of the electric field). 
This results in a second-order integral-differential equation with a weakly singular  integral kernel  expressed by  the Mittag-Lefler (ML) function. The most interesting but challenging issue resides in how to efficiently  deal with the singularity in time induced by the ML function which is an infinite series of singular power functions  with different nature.  
  With this in mind,  we  introduce a spectral-Galerkin method using Fourier-like basis functions for spatial discretization, leading to  a sequence of  decoupled temporal integral-differential equations (IDE) with  the same weakly singular kernel involving the ML function as the original two-dimensional problem.  With a careful study of  the regularity of IDE, we incorporate several leading singular terms into the 
 numerical scheme and approximate much regular part of the solution.  Then we solve to IDE by a  multi-step well-conditioned collocation scheme together with mapping technique to increase the accuracy and enhance the resolution.  We show such an enriched collocation method is  convergent and accurate. 
\end{abstract}
	

	
\maketitle


\section{Introduction}

\setcounter{equation}{0}
\setcounter{lmm}{0}
\setcounter{thm}{0}

In electromagnetism, if the electric  permittivity or magnetic  permeability
depends on the wave frequency, then the medium is called a dispersive medium.
The typical  models that characterise such a dependence include the  Drude mode \cite{Ziolkowski03, ZiolkowskiH01} and the Lorenz model \cite{ShelbySNS01, SmithK00}.  The Cole-Cole (C-C) dispersive model, distinguishing itself by the nonlocal feature, has been successfully  applied to fit experimental dispersion and absorption for
a considerable number of liquids and dielectrics
\cite{ColeC41}. Such a model can be expressed by the empirical formula (cf. \cite{ColeC41}): 
\begin{equation}
\epsilon(\omega)=\epsilon_0\bigg(\epsilon_{\infty}+\frac{\epsilon_s-\epsilon_{\infty}}
{1+({\rm i}\,\omega\tau)^{\alpha}}\bigg), \quad  0<\alpha\leq 1,
\end{equation}
where $\tau, \epsilon_0, \epsilon_s, \epsilon_\infty$ are all given physics constants.
Here,  $\tau$ is the central relaxation time of the material model, $\epsilon_0$ is the permittivity of vacuum, and  $\epsilon_{s}$ and $\epsilon_\infty$  are respectively the zero- and infinite-frequency limits of the relative permittivity satisfying $\epsilon_s>\epsilon_{\infty}\geq 1$. In particular, the model with $\alpha = 1$ leads to the classical Debye dielectric model, or exponential dielectric relaxation. 

Since the C-C relaxation model has many applications in diverse fields, such as soil characterization \cite{RepoP96},  permittivity of biological tissue \cite{GabrielG96}, and the transient nature of electromagnetic radiation in the human body \cite{CooperMBH00, KimKP07} and among others,  its numerical solution has attracted much attention. 
Intensive studies have been devoted to the finite difference time domain (FDTD) methods (cf. \cite{Causley11, RecanosP10, RecanosY14, Tofighi09, TorrresVJ96}), and the time-domain finite element methods  \cite{BanksBG09, JiaoJ01,LuZC04, WangXZ10,LiHL11}. 
Most of them 
worked on discretization of the Maxwell system directly 
 where the electric field and the induced electric polarisation in the model are interconnected  and globally dependent (see \eqref{Cole-Cole}).  Although this relation can be transformed into a fractional differential equation (see, e.g., \cite{TorrresVJ96, LiHL11, RecanosY14}),  direct discretisation of three fields may result in large degree of freedoms with a heavy burden  of historical dependence in time. 


Different from all aforementioned works, we  formulate  the C-C model  as a  second-order partial integral-differential 
equation (PIDE)  involving only one unknown field, where the integral part has a weakly singular kernel in terms of  the ML function.  
We then place the emphasis on how to efficiently deal with the temporal singular integral with the kernel function as a series of singular functions in different fractional powers. Without loss of generality, we consider the plane wave geometry of the C-C model and reduce magnetic and electric field vectors to scalar field quantities by polarisation, and restrict our attention to the two-dimensional PIDE.  
  We then employ a spectral-Galerkin method  using Fourier-like basis functions  in space (cf. \cite{Liu.S96,Liu.T01,She.W07b}), and the model 
 boils down to  a sequence of decoupled temporal  IDE with the same type of singular integrals. 
   As such, unlike the existing  methods, we work with a model  with the  minimum number of unknowns, so the computational cost  can be enormously reduced.   

We propose a well-conditioned multi-step collocation method for solving the temporal IDE, which is enriched by incorporating  a few leading singular terms through a delicate regularity analysis, and  integrated with   a mapping technique  (cf.  \cite{WangS05})    for treating the singular integral and nearly singular integrals  in the first subinterval. The well-conditioning is achieved by writing the IDE in a first-order damped Hamitonian system and using  the Birkhoff-Lagrange interpolating basis  (cf. \cite{WangSZ14}), so the proposed method possesses a long time stability. 
It is noteworthy that the integral operator in our setting involving the ML function as the singular kernel.  
Such a kernel  
is distinct from the usual weakly singular kernel,  such as 
 $t^{\alpha-1},\, 0<\alpha<1,$ in terms of the singular behaviors.  We notice  that many fast algorithms, stemming from the celebrated Fast Multipole method, have been recently  proposed for the (Riemann-Liouville/Caputo) fractional differential equations (see, e.g., \cite{LiR10, McLean12, JiangZZZ17}).  However, it appears that the extension of  these algorithms  to our  case  is nontrivial and largely open due to the completely different nature of the singular kernel. 


The rest of the paper is organised as follows. In Section \ref{Sect2}, we  formulate our new model and present a semi-discretised  scheme for  the problem of interest. In Section \ref{Sect3}, we tackle 
the challenges of the  temporal IDE  obtained from the previous section, and introduce effective numerical techniques to surmount the obstacles. We also present various numerical results to illustrate various perspectives of the proposed method.  
We then conclude with discussions and some future work regarding the C-C model in Section \ref{Sect4}.

\section{Formulation of the model and a semi-discretised scheme}\label{Sect2}
\setcounter{equation}{0}
\setcounter{lmm}{0}
\setcounter{thm}{0}

In this section, we derive a new model from the Maxwell system in Cole-Cole media involving three vector fields, and introduce 
a semi-discretised scheme  for  the problem of interest. 
\subsection{Maxwell's equations in Cole-Cole media}

The time-domain Maxwell's equations in a Cole-Cole media take the form (cf. \cite{TorrresVJ96,LiHL11}): 
\begin{subequations}\label{22eqn}
\begin{align}
\epsilon_0\epsilon_{\infty}\frac{\partial {\bs E}}{\partial t}&=\nabla\times{\bs H}-\frac{\partial {\bs P}}{\partial t} \quad\;\;  {\rm in}\;\; \Omega\times(0, T],\label{cc1}\\
\mu_0\frac{\partial {\bs H}}{\partial t}&=-\nabla\times{\bs E} \qquad\qquad  {\rm in} \;\; \Omega\times(0, T],\label{cc2}
\end{align}
\end{subequations}
where $\Omega$ is a bounded domain in ${\mathbb R}^3$ with a Lipschitz boundary, and  ${\bs P}({\bs x}, t)$ is the induced electric polarization 
\begin{equation}
\label{Cole-Cole}
{\bs P}(\bs x, t)=\int_0^t \xi_{\alpha}(t-s){\bs E}(\bs x,s)\,ds,\quad \xi_{\alpha}(t):=\mathscr{L}^{-1}\Big\{\frac{\epsilon_0(\epsilon_s-\epsilon_{\infty})}
{1+(s\tau)^{\alpha}}\Big\}.
\end{equation}
 Here, $\xi_{\alpha}$ is  the time-domain susceptibility kernel which involves the inverse Laplace transform  $\mathscr{L}^{-1}$. Note that  ${\bs P}(\bs x, 0)=0$ is evident from \eqref{Cole-Cole}.
Here, we supplement \eqref{22eqn}-\eqref{Cole-Cole} with the initial condition
\begin{equation}\label{CC:i}
\bs E(\bs x,0)=\bs E_0(\bs x),\quad \bs H(\bs x,0)=\bs H_0(\bs x) \quad {\rm in}\;\; \Omega, 
\end{equation}
and a perfect conducting boundary condition
\begin{equation}\label{CC:b}
\bs n\times \bs E=\bs 0\quad {\rm at}\;\;
 \partial \Omega \times (0,T).
\end{equation}



It is seen that the above Maxwell's system contains three unknown vector fields.
It is computationally beneficial to eliminate some unknowns. In this paper, we work on the model with  one unknown field. More precisely, we denote
\begin{equation}\label{constA}
a={\frac{1}{\mu_0\epsilon_0\epsilon_\infty}}, \quad b=\frac{\epsilon_s-\epsilon_\infty}{\mu_0\tau^\alpha\epsilon_0\epsilon_\infty^2}, \quad \lambda=\frac{\epsilon_s}{\epsilon_\infty\tau^\alpha},
\end{equation}
and  define
\begin{equation}\label{AuxiF}
{\bs {\mathcal E}}=\epsilon_0\epsilon_{\infty}{\bs E}+\bs P.
\end{equation}
Then we can derive from \eqref{22eqn}-\eqref{Cole-Cole} the  integro-differential equation:
\begin{equation}\label{modeqnA}
\begin{split}
\frac{\partial^2 \bs {\mathcal E}}{\partial t^2}=-a\nabla\times\nabla\times \bs {\mathcal E}+b\int_0^t(t-s)^{\alpha-1}E_{\alpha,\alpha}(-\lambda(t-s)^\alpha)\nabla\times\nabla\times \bs {\mathcal E}(\bs x,s)ds,
\end{split}
\end{equation}
for $0<\alpha<1,$ where $E_{\alpha,\beta}(t)$
 is the standard Mittag-Leffler (ML) function defined by (cf. \cite{GorenfloK14}):
 \begin{equation}\label{MLfuncA}
E_{\alpha,\beta}(z)=\sum\limits_{k=0}^\infty \frac{z^k}{\Gamma(k\alpha+\beta)}.
\end{equation}
Now, we show the derivation of \eqref{modeqnA}-\eqref{MLfuncA}. Firstly, taking
 derivative with respect to $t$ for \eqref{cc1} and $\nabla\times $ for \eqref{cc2}, we eliminate $\bs H$ and obtain
\begin{equation}
\label{wave}
\epsilon_0\epsilon_{\infty}\frac{\partial^2 {\bs E}}{\partial t^2}=-\frac{1}{\mu_0}\nabla\times\nabla\times {\bs E}-\frac{\partial^2 {\bs P}}{\partial t^2}. 
\end{equation}
Secondly, taking the Laplace transform on both sides of  \eqref{Cole-Cole} leads to
\begin{equation}\label{eq:FDE1}
{\widehat{\bs P}}(\bs x,s)=\frac{\epsilon_0(\epsilon_s-\epsilon_{\infty})}{1+(s\tau)^\alpha}{\widehat{\bs E}}(\bs x, s),
\end{equation}
where the notation $\widehat{\bs W}$ stands for  the Laplace transform of the field $\bs W$.  Then a direct calculation from \eqref{AuxiF} and \eqref{eq:FDE1} yields
 \begin{equation}\label{newderi}
 \begin{split}
\widehat{\bs E}&= \frac{1+(s\tau)^\alpha}{\epsilon_0\epsilon_\infty(1+(s\tau)^\alpha)+\epsilon_0(\epsilon_s-\epsilon_\infty)}
\widehat{\bs {\mathcal E}}=\frac{1}{\epsilon_0\epsilon_\infty}
\widehat{\bs {\mathcal E}}-\frac{\epsilon_s-\epsilon_\infty}{\epsilon_0\epsilon_\infty^2\tau^\alpha}\frac{1}{s^\alpha+\epsilon_s/(\epsilon_\infty\tau^\alpha)}
\widehat{\bs {\mathcal E}}.
\end{split}
\end{equation}
Recall the formula of the inverse Laplace transform \cite[p. 84]{GorenfloK14}:
\begin{equation}\label{invLap}
\mathcal{L}^{-1}\Big(\frac{1}{s^\alpha+\lambda}\Big)= t^{\alpha-1}E_{\alpha,\alpha}(-\lambda t^\alpha), \quad{\rm if}\;\; |\lambda/s^\alpha|<1.
\end{equation}
Applying the inverse Laplace transform on both sides of \eqref{newderi} and using \eqref{invLap}, we obtain
\begin{equation}\label{eq:EF}
\bs E=\frac{1}{\epsilon_0\epsilon_\infty}\bs {\mathcal E}-\frac{\epsilon_s-\epsilon_\infty}{\epsilon_0\epsilon_\infty^2\tau^\alpha}\int_0^t (t-s)^{\alpha-1}E_{\alpha,\alpha}\Big(-\frac{\epsilon_s}{\epsilon_\infty\tau^\alpha}(t-s)^\alpha\Big) \bs {\mathcal E} (\bs x, s)ds.
\end{equation}
Substituting \eqref{eq:EF} and $\bs P={\bs {\mathcal E}}-\epsilon_0\epsilon_{\infty}{\bs E}$ into \eqref{wave}
leads to \eqref{modeqnA}.

With the substitution \eqref{AuxiF}, we can determine the initial and boundary conditions of the new known field
$\bs {\mathcal E}$ as follows.  By \eqref{Cole-Cole} and \eqref{CC:b}, we have $\bs n\times \bs P=\bs 0$ at the boundary, so
\begin{equation}\label{pecbnd}
\bs n\times \bs {\mathcal E}=\bs 0 \quad   {\rm at}\;\;
 \partial \Omega \times (0,T),
\end{equation}
and similarly, we can derive the initial conditions from \eqref{cc1}, \eqref{CC:i} and \eqref{AuxiF} as follows
\begin{equation}\label{inicondA}
\bs {\mathcal{E}}(\bs x, 0)=\epsilon_0\epsilon_\infty \bs E_0(\bs x):=\bs {\mathcal E}_0(\bs x), \quad \bs {\mathcal{E}}_t(\bs x, 0)=\nabla\times \bs H_0(\bs x)
:=\bs {\mathcal E}_1(\bs x)\quad {\rm in}\;\; \Omega
\end{equation}

In summary, with the aid of the auxiliary field $\bs {\mathcal E}$ in \eqref{AuxiF}, we can  reformulate  the Cole-Cole model \eqref{22eqn}-\eqref{CC:b} as the following integro-differential model with one unknown field:
\begin{equation}\label{Reformu:CC}
\begin{cases}
\dfrac{\partial^2 \bs {\mathcal E}}{\partial t^2}=-a\nabla\times\nabla\times \bs {\mathcal E}+b\displaystyle\int_0^te_{\alpha,\alpha}(-\lambda(t-s)^\alpha) \nabla\times\nabla\times \bs {\mathcal E}\,ds
\; & {\rm in}\;\; \Omega,\; t\in (0,T],\\[6pt]
\bs n\times \bs {\mathcal E}|_{\partial\Omega}=0\ \ \  & t\in (0,T],\\[4pt]
\bs {\mathcal{E}}(\bs x, 0)=\bs {\mathcal E}_0(\bs x), \quad \bs {\mathcal{E}}_t(\bs x, 0)
=\bs {\mathcal E}_1(\bs x)  \quad & {\rm in}\  \Omega,
\end{cases}
\end{equation}
where $0<\alpha<1,$ $a, b,\lambda$ are given by \eqref{constA}, and
\begin{equation}\label{eafaf}
e_{\alpha,\alpha}(-\lambda t^\alpha)=t^{\alpha-1}E_{\alpha,\alpha}(-\lambda t^\alpha).
\end{equation}
Note that we can recover the electric field from $\bs{\mathcal E}$  by \eqref{eq:EF}:
\begin{equation}\label{eq:EF2}
\bs E(\bs x, t)=a \mu_0\, \bs {\mathcal E}(\bs x,t)-b\mu_0 \int_0^t e_{\alpha,\alpha}(-\lambda(t-s)^\alpha) \bs {\mathcal E} (\bs x, s)ds.
\end{equation}

\subsection{Two-dimensional Cole-Cole model}
It is seen from \eqref{modeqnA} that the most interesting but challenging  issue lies in the treatment of the singular integral
in time. Without loss of generality, we consider the transverse electric polarization with
$\bs E=(0,0,E_z(x,y))',$ so we have $\bs{\mathcal{E}}=(0,0,u(x,y))'.$
Then we have the reduced model of \eqref{modeqnA}:
%
%
\begin{equation}\label{Prob:CC}
\begin{cases}
\displaystyle{\frac{\partial^2 u}{\partial t^2}=a\Delta u-b\int_0^te_{\alpha,\alpha}(-\lambda(t-s)^\alpha)\Delta u(\bs x, s)ds}\; & {\rm in}\;\; \Omega,\; t\in (0,T],\\[6pt]
 u(\bs x, t)|_{\partial\Omega}=0\ \ \  & {\rm for}\;\; t\in (0,T],\\[4pt]
 u(\bs x, 0)=u_0(\bs x), \quad  u_t(\bs x, 0)= u_1(\bs x)  \quad & {\rm in}\  \Omega.
\end{cases}
\end{equation}



The existence and uniqueness of a weak solution to \eqref{Prob:CC} has been investigated in  \cite{LarssonF10} by a semigroup approach and
further explored in \cite{Saed14} using the classic energy argument. However, both studies require $u_0\in H^1(\Omega)$.
In what follows, we shall show $L^2$-{\it a priori} stability with a minimum requirement of the regularity, that is,  $ u_0\in L^2(\Omega)$, which is accomplished by following the spirit  of \cite{Baker76}. 
\begin{theorem}\label{thm:energy}
Let $ u$ be the solution of \eqref{Prob:CC}. If $u_0,u_1\in L^2(\Omega)$ and $a-b/ \lambda \geq 0$, then we have $u\in  L^\infty(0,T; L^2(\Omega)) $ and   the following  estimate
\begin{equation}
 \|u\|_{L^\infty(0,T; L^2(\Omega))} 
 \leq \sqrt 2\|u_0\|_{L^2(\Omega)}+2T\| u_1\|_{L^2(\Omega)}.\label{energy1}
 \end{equation}
\end{theorem}
\begin{proof} 
Setting
\begin{equation}\label{eq:phi}
 \phi(\bs x,t)=\int_t^\xi  u(\bs x,\theta)d\theta, \quad  \xi\in [0,T],
\end{equation}
one verifies easily that
$$
 \phi(\bs x,\xi)=0, \quad  \frac{\partial \phi}{\partial t}(\bs x,t)=-u(\bs x,t).
$$
Multiplying both sides of the first equation in \eqref{Prob:CC} by $\phi(\bs x,t)$  and integrating in space over $\Omega$, we have
\begin{equation}\label{eq:uphi0}
 ( u_{tt}, \phi)=-a(\nabla  u,\nabla \phi)+b\int_0^t e_{\alpha,\alpha}(-\lambda(t-s)^\alpha)(\nabla  u,\nabla\phi)ds.
\end{equation}
Further integrating both sides with respect to $t$ over $(0,\xi)$ leads to
\begin{equation}\label{eq:uphi}
 \int_0^\xi (u_{tt}, \phi)dt=-a \int_0^\xi (\nabla  u,\nabla \phi)dt+b \int_0^\xi \int_0^t e_{\alpha,\alpha}(-\lambda(t-s)^\alpha)(\nabla  u,\nabla\phi)dsdt.
\end{equation}
Next, using integration by parts and the explicit form of $ \phi$ in  \eqref{eq:phi}, we find
\begin{equation}\label{eq:1pos}
\begin{split}
\int_0^\xi ( u_{tt},\phi)dt& =\int_\Omega \Big((u_t\phi)\Big|_0^\xi-\int_0^\xi u_td\phi\Big)dxdy
    \\&=\frac{1}{2}\|u(\cdot,\xi)\|_{L^2(\Omega)}^2-\frac{1}{2}\|u_0\|_{L^2(\Omega)}^2-\int_\Omega u_1(\bs x)\phi(\bs x,0)dxdy,
\end{split}
\end{equation}
and
%
\begin{equation}\label{eq:1posB}
\begin{split}
\int_0^\xi (\nabla u,\nabla\phi)dt&=\int_0^\xi\int_t^\xi (\nabla  u(\cdot, t),\nabla  u(\cdot, \theta)) d\theta dt=\int_0^\xi\int_0^\theta (\nabla  u(\cdot,t),\nabla  u(\cdot,\theta))dt d\theta\\&=\frac{1}{2}\int_\Omega\Big|\int_0^\xi\nabla u(\bs x, t)dt\Big|^2dxdy,
\end{split}
\end{equation}
where in the last step, we used the property:
\begin{equation}\label{gttheta0}
\begin{split}
&\int_0^\xi\int_0^\theta g(t)g(\theta)dt d\theta= \int_0^\xi\int_t^\xi g(t)g(\theta) d\theta dt
= \int_0^\xi\int_\theta^\xi g(t)g(\theta) dt d\theta,
\end{split}
\end{equation}
implying
\begin{equation}\label{gttheta1}
\begin{split}
&\int_0^\xi\int_0^\theta g(t)g(\theta)dt d\theta= \frac 1 2 \int_0^\xi\int_0^\xi g(t)g(\theta) d\theta dt
= \frac 1 2 \Big|\int_0^\xi g(t) dt\Big|^2.
\end{split}
\end{equation}

Now, we deal with the  singular integral term in \eqref{eq:uphi}.
 It is straightforward to verify from the definition \eqref{MLfuncA}  that
\begin{equation}\label{Eaf1prop}
\begin{split}
\frac{1}{\lambda}\frac{d}{dt}E_{\alpha,1}(-\lambda(t-s)^\alpha)&=\frac{1}{\lambda}\frac{d}{dt}\sum\limits_{k=0}^\infty\frac{(-\lambda)^k(t-s)^{k\alpha}}{\Gamma(k\alpha+1)}=\frac{1}{\lambda}\sum\limits_{k=0}^\infty\frac{(-\lambda)^{k+1}(t-s)^{k\alpha+\alpha-1}}{\Gamma(k\alpha+\alpha)}
\\&=- (t-s)^{\alpha-1}E_{\alpha,\alpha}(-\lambda(t-s)^\alpha)=-e_{\alpha,\alpha}(-\lambda (t-s)^\alpha).
\end{split}
\end{equation}
Using the above property, we derive   
\begin{equation}\label{eq:1posC}
\begin{split}
\int_0^\xi\int_0^te_{\alpha,\alpha} & (-\lambda(t-s)^\alpha)\big(\nabla  u(\cdot, s),\nabla\phi(\cdot, s)\big)dsdt
\\& =\int_0^\xi\int_s^\xi e_{\alpha,\alpha} (-\lambda(t-s)^\alpha)\big(\nabla u(\cdot,s),\nabla\phi(\cdot, s)\big)dtds\\
&=\Big(-\frac{1}{\lambda}\Big)\int_0^\xi\Big\{\int_s^\xi  d E_{\alpha,1}(-\lambda(t-s)^\alpha)\Big\}
\big(\nabla u(\cdot, s),\nabla\phi(\cdot,s)\big)ds\\
&=\frac{1}{\lambda}\int_0^\xi \big\{1-  E_{\alpha,1}(-\lambda(\xi-s)^\alpha)\big\} \big(\nabla u(\cdot, s),\nabla\phi(\cdot,s)\big)ds.
\end{split}
\end{equation}
Hence, we obtain from the above identities  that 
\begin{equation}\label{eq:u1}
\begin{split}
\frac{1}{2}\| u(\cdot,\xi)\|_{L^2(\Omega)}^2&+\Big(\frac{a}{2}-\frac{b}{2\lambda}\Big)\int_\Omega \Big|\int_0^\xi\nabla u(\bs x, t)dt\Big|^2dxdy\\
&=\frac{1}{2}\|u_0\|_{L^2(\Omega)}^2+\int_\Omega u_1(\bs x)\phi(\bs x,0)dxdy\\
&\quad  -\frac{b}{\lambda}\int_0^\xi  E_{\alpha,1}(-\lambda(\xi-s)^\alpha) (\nabla u(\cdot,s),\nabla\phi(\cdot,s))ds.
\end{split}
\end{equation}
Note that $0<E_\alpha(-\lambda t^\alpha)\le 1$ and it is monotonically decreasing (cf. \cite{LarssonF10, Saed14}). In view of the second mean value theorem \cite{WitulaH12}, there exists  $\xi_0\in (0,\xi)$ such that
\begin{equation}\label{eq:u12}
\begin{split}
&\int_0^\xi  E_{\alpha,1}(-\lambda(\xi-s)^\alpha) (\nabla u(\cdot, s),\nabla\phi(\cdot, s))ds=
\int_{\xi_0}^\xi (\nabla u(\cdot,s),\nabla\phi(\cdot,s))ds\\
&\quad
=\int_{\xi_0}^\xi\int_s^\xi (\nabla u(\cdot, s),\nabla u(\cdot, \theta))d\theta ds
=\frac{1}{2}\int_\Omega\Big|\int_{\xi_0}^\xi\nabla u(\bs x, t)dt\Big|^2dxdy\geq 0,
\end{split}
\end{equation}
where we used \eqref{gttheta0}-\eqref{gttheta1}.
Therefore, by \eqref{eq:phi}, \eqref{eq:u1} and the Cauchy-Schwarz inequality, 
\begin{align*}
&\frac{1}{2}\| u(\cdot, \xi)\|_{L^2(\Omega)}^2+\Big(\frac{a}{2}-\frac{b}{2\lambda}\Big)\int_\Omega \Big|\int_0^\xi\nabla u(\bs x, t)dt\Big|^2dxdy
 \\& \leq \frac{1}{2}\| u_0\|_{L^2(\Omega)}^2+\int_\Omega  u_1(\bs x)\phi(\bs x,0)dxdy= \frac{1}{2}\| u_0\|_{L^2(\Omega)}^2+\int_0^\xi \int_\Omega  u_1(\bs x) u(\bs x,\theta)dxdy d\theta\\
&\leq \frac{1}{2}\| u_0\|_{L^2(\Omega)}^2+\| u_1\|_{L^2(\Omega)}\int_0^\xi \| u(\cdot, \theta)\|_{L^2(\Omega)}d\theta\le
 \frac{1}{2}\| u_0\|_{L^2(\Omega)}^2+T\| u_1\|_{L^2(\Omega)}  \| u\|_{L^\infty(0,T;L^2(\Omega))}.
\end{align*}
Therefore,  if ${a}-{b}/{\lambda}\geq 0,$ then by the Cauchy-Schwarz inequality,
\begin{align*}
\frac 1 2 \| u\|_{L^\infty(0,T;L^2(\Omega))}^2 &\leq \frac 1 2  \|u_0\|_{L^2(\Omega)}^2+T\| u_1\|_{L^2(\Omega)}\|u\|_{L^\infty(0,T;L^2(\Omega))} \notag\\
  &\leq \frac 1 2 \|u_0\|_{L^2(\Omega)}^2+\frac{1}{4}\|u\|_{L^\infty(0,T;L^2(\Omega))}^2+ {T^2}  \| u_1\|_{L^2(\Omega)}^2,
\end{align*}
which immediately implies \eqref{energy1}.
\end{proof}

\begin{rem}\label{bond} {\em
Using a standard energy argument, we can  follow \cite{CannarsaS08} to derive the estimate:
\begin{equation}\label{existest}
\begin{split}
\|u_t\|^2_{L^\infty(0,T;L^2(\Omega))}&+(a-b/\lambda)\|\nabla u\|^2_{L^\infty(0,T;L^2(\Omega))}\\
& \leq \|u_1\|^2_{L^2(\Omega)}+\Big(a+\frac{b}{\lambda}+\frac{2b^2}{(a\lambda-b)\lambda}\Big)\|\nabla u_0\|^2_{L^2(\Omega)},
\end{split}
\end{equation}
under the condition: $a-b/\lambda\ge 0.$ }
\end{rem}



%
%

\subsection{Spectral-Galerkin discretization using Fourier-like basis in space}
As we are mostly interested in dealing with the singular fractional integrals, we consider $\Omega=(-1,1)$ or $\Omega=(-1,1)^2.$ Let ${\mathbb P}_N$ be the set of all polynomials of degree at most $N,$ and let ${\mathbb P}_N^0=\big\{\phi\in {\mathbb  P}_N\,:\, \phi=0\; {\rm on}\; \partial \Omega\big\}.$
The spectral-Galerkin approximation of \eqref{Prob:CC} in space is to find $u_N(\cdot, t)\in {\mathbb P}_N^0$ such that   for any  $v_N,w_N, z_N\in {\mathbb P}_N^0,$
\begin{equation}\label{eq:scheme}
\begin{cases}
 (\partial_t^2 u_N, v_N)_{\Omega}+a(\nabla  u_N,\nabla v_N)_{\Omega}=b\displaystyle \int_0^t e_{\alpha,\alpha}(-\lambda(t-s)^\alpha)(\nabla  u_N,\nabla v_N)_{\Omega}\,ds,
 \\[8pt]
  (u_N(\cdot, 0),w_N)_{\Omega}=(u_{0},w_N)_{\Omega},\quad
  (\partial_t u_N(\cdot, 0),w_N)_{\Omega}=(u_{1},z_N)_{\Omega}. 
 \end{cases}
\end{equation}
We next employ the matrix diagonalization technique (cf. \cite[Ch. 8]{ShenTW11}) to reduce  \eqref{eq:scheme} to a sequence of
integral-differential equations in time.

 We first look at the one-dimensional case. Define
\begin{equation}\label{basisA}
\phi_k(x)=\frac{1}{\sqrt{4k+6}}(L_k(x)-L_{k+2}(x)),  \quad k\geq 0,
\end{equation}
where $L_k(x)$ is the  Legendre polynomial of degree $k$. Then we have
\begin{equation}\label{PN0}
{\mathbb P}_N^0={\rm span}\{\phi_k:\; 0\le k\le N-2\}.
\end{equation}
 It is known that under this basis, the stiffness matrix is identity as  $(\phi'_k,\phi'_j)=\delta_{kj},$
 and the mass matrix
$B$ with entries $b_{kj}=(\phi_k,\phi_j)_{\Omega}$ is symmetric and pentadiagonal (cf. \cite{Shen94b}).  
Thus, writing
\begin{equation}\label{eq:1DSu}
u_N(x,t)=\sum\limits_{k=0}^{N-2} \hat u_k(t) \phi_k(x),\quad \bs {\hat u}(t) =(\hat u_{0}(t),  \hat u_1(t),\cdots, \hat u_{N-2}(t))',
\end{equation}
 the scheme \eqref{eq:scheme} becomes
\begin{equation}\label{eq:coupled}
\begin{cases}
B\bs {\hat  u}''(t)+a \bs {\hat  u}(t)=b\displaystyle\int_0^t e_{\alpha,\alpha}(-\lambda(t-s)^\alpha)\bs {\hat  u}(s) ds,\quad t\in (0,T],\\[8pt]
 B\bs {\hat  u}(0)=\bs {\hat u}_0,\quad  B \bs {\hat  u}'(0)=\bs {\hat u}_1,
\end{cases}
\end{equation}
where $\bs {\hat u}_i=((u_i,\phi_0)_\Omega,\cdots, (u_i,\phi_{N-2})_\Omega)'$ for $i=0,1.$ Let $\{\lambda_i\}_{i=0}^{N-2}$ be the eigenvalues of $B,$ and let $E$ be the corresponding eigenvectors of $B.$ Note that $E$ is an orthonormal matrix, so $E'E=I_{N-1}.$
Introducing the change of variables: $\bs {\hat u}=E \bs {v}$ with $\bs v=(v_0,v_1,\cdots, v_{N-2})',$   we can decouple the system \eqref{eq:coupled} into
\begin{equation}\label{eq:1D}
\begin{cases}
{v}_i''(t)+{a}\lambda_i^{-1} v_i(t)={b}\lambda_i^{-1}\displaystyle\int_0^t e_{\alpha,\alpha}(-\lambda (t-s)^\alpha)v_i(s)ds, \quad t\in (0,T], \\[8pt]
v_i(0)=\lambda_i^{-1}v_{0i},\quad  v_i'(0)=\lambda_i^{-1}v_{1i},
\end{cases}
\end{equation}
for $i=0,\cdots, N-1,$ where   $\bs {\hat u}_j=E \bs {v}_j$ with $\bs v_j=(v_{j0},v_{j1},\cdots, v_{j(N-2)})'$ for $j=0,1.$

Similarly, in the two-dimensional case, we have
\begin{equation}\label{2Dcase}
{\mathbb P}_N^0=\text{span}\big\{\phi_i(x)\phi_j(y)\;:\; 0\leq i,j\leq N-2\big\}.
\end{equation}
We write
\begin{align}\label{2Dusolu}
& u_N(x,t)=\sum\limits_{i,j=0}^{N-2} \hat{u}_{ij}(t)\phi_i(x)\phi_j(y), \quad  {\widehat U}(t)=(\hat{u}_{ij}(t))_{i,j=0,\cdots, N-2},
\end{align}
Then the counterpart of \eqref{eq:coupled} becomes
\begin{equation}\label{eq:coupled2D}
\begin{cases}
B {\widehat  U}''B+a ({\widehat  U}B+B {\widehat  U})=b\displaystyle\int_0^t e_{\alpha,\alpha}(-\lambda(t-s)^\alpha)({\widehat  U}B+B {\widehat  U}) ds,\quad t\in (0,T],\\[8pt]
 B{\widehat  U}B|_{t=0}=\widehat U_0,\quad   B{\widehat  U}'B|_{t=0}=\widehat U_1,
\end{cases}
\end{equation}
Using the full matrix diagonalisation technique and  setting 
 ${\widehat U}=EWE'$ with $W=(w_{ij})_{i,j=0\cdots,N-2}$ (cf. \cite[Ch. 8]{ShenTW11}), we  have
\begin{equation}\label{eq:2D}
\begin{cases}
{w}_{ij}''(t)+a\big({\lambda_i^{-1}+\lambda_j^{-1}}\big)w_{ij}(t)=b\big({\lambda_i^{-1}+\lambda_j^{-1}}\big) \displaystyle \int_0^t e_{\alpha,\alpha}(-\lambda (t-s)^\alpha)w_{ij}(s)ds, \\[8pt]
w_{ij}(0)=(\lambda_i\lambda_j)^{-1}w_{ij}^0,\quad  v_i'(0)=(\lambda_i\lambda_j)^{-1}w_{ij}^1,
\end{cases}
\end{equation}
for all  $t\in (0,T]$, where
${\widehat U}_k=EW^kE'$ and $W^k=(w_{ij}^k)_{i,j=0,\cdots,N-2}$ for $k=0,1$.

\section{Algorithm development for the integral-differential equation}\label{Sect3} 
\setcounter{equation}{0}
\setcounter{lmm}{0}
\setcounter{thm}{0}

  \subsection{Prototype problem} Consider the prototype integral-differential  equation:
\begin{equation}\label{eq:1T}
 \begin{cases}
 u''(t)+c u(t)=d\displaystyle\int_0^t e_{\alpha,\alpha}(-\lambda (t-s)^\alpha)u(s)ds,\quad  t\in (0,T], \;\;\; 0<\alpha<1,\\[7pt]
  u(0)=u_0, \quad u'(0)=u_1,
 \end{cases}
 \end{equation}
 where the constants $c,d>0,$ and the singular kernel $e_{\alpha,\alpha}(t)=t^{\alpha-1}E_{\alpha,\alpha}(t)$ (cf. \eqref{eafaf}).

To alleviate ill-conditioning of the following multistep collocation method, we  adopt an ingredient of numerical treatment for Hamiltonian systems (cf. \cite{FengQ03})   and rewrite  \eqref{eq:1T} into the first-order system:   
\begin{equation}\label{pqsystem}
\begin{cases}
  {p}'(t)+cq(t)=d\displaystyle\int_0^t e_{\alpha,\alpha}(-\lambda (t-s)^\alpha) q(s)ds;  
 \quad {q}'(t)=p(t),  \quad t\in (0,T],\\[6pt]
 q(0)=u_0, \quad p(0)=u_1,
\end{cases}
\end{equation}
by setting $q=u$ and $p=u'.$

Similar to Theorem \ref{thm:energy}, we have the following stability of \eqref{pqsystem}.
\begin{theorem}\label{thm:1D}
Assume  $u_0=0$ and $a-b/\lambda\geq 0$ in \eqref{pqsystem}. Then, we have the bound
\begin{equation}
p^2(t)+(a-b/\lambda)q^2(t)\leq u_1^2, \ \ \ \forall t\in [0,T].
\end{equation}
\end{theorem}
\begin{proof}
The proof is the same as that of Theorem \ref{thm:energy}, and hence is omitted.
 \end{proof}
\begin{rem}\label{incase} {\em
If $a-b/\lambda\geq 0$, one can define a Hamiltonian
\begin{equation}\label{Htform}
H(t)=(u'(t))^2+(a-b/\lambda)u^2(t),
\end{equation}
for \eqref{eq:1T} and obtain a damped Hamiltonian system.

The assumption $u_0=0$ seems restrictive, however, it is indispensable for this bound. Our numerical experiments show that the Hamiltonian may increases or even outweighs the initial Hamiltonian without the condition  {\rm(}see Figure \ref{Fig:ExODE1} below{\rm)}. }
\end{rem}

\subsection{A multistep collocation method}
For simplicity, we partition the interval $[0,T]$ into $K$  subintervals of equal length, that is,
$$
I_k=(t_{k-1},t_k), \quad  t_k=kT/K,\quad k=1,\cdots, K;\quad t_0=0.$$
Let $\{x_j\}_{j=0}^N\subseteq [-1,1]$ be a set of Jacobi-Gauss-Lobatto (JGL) points arranged in ascending order, and denote the grids
\begin{equation}\label{Ijnodes}
t_j^{k}=\frac{t_{k-1}+t_k} 2+\frac{t_k-t_{k-1}} 2 x_j,\quad 0\le  j\le N;\quad 1\le k\le K.
\end{equation}
 Let $P_N, Q_N \in C^0(0,T)$ be the   multistep spectral-collocation approximations of $p,q,$ respectively,
 and each
consists of  $K$ pieces:
\begin{equation}\label{approx1}
\begin{split}
& P_N|_{I_1}=p_N^1=p_*+\hat p_N^1, \quad Q_N|_{I_1}=q_N^1=q_*+\hat q_N^1, \quad \hat p_N^1, \hat q_N^1 \in {\mathbb P}_N;\\
& P_N|_{I_k}=p_N^k\in {\mathbb P}_N, \quad Q_N|_{I_k}=q_N^k\in {\mathbb P}_N, \quad  k=2,3\cdots, K,
\end{split}
\end{equation}
where  $p_*,q_*$ are two pre-defined functions to capture leading singular terms (see Subsection \ref{pqdefn}).

  We find these $K$ pieces  in sequence as follows.
\begin{itemize}
\item For $k=1,$ we find  $\{p_N^1, q_N^1\}$ via the collocation scheme:
\begin{equation}\label{k1schm}
\begin{cases}
  \dot{p}_{N}^1(t_j^1)+cq_{N}^1(t_j^1)=d\displaystyle\int_0^{t_j^1} e_{\alpha,\alpha}(-\lambda (t_j^1-s)^\alpha)
  q_{N}^1(s)ds, \quad 1\le j\le N; \\[7pt]
 \dot{q}_{N}^1(t_j^1)=p_{N,1}(t_j^1),\quad 1\le j\le N; \\[7pt]
 q_{N}^1(0)=u_0, \quad p_{N}^1(0)=u_1,
\end{cases}
\end{equation}

\vskip 4pt
\item For any $k\in \{2,\cdots, K\},$ using the   computed values $\{p_N^l, q_N^l\}_{l=1}^{k-1},$
we find $\{p_N^k, q_N^k\}$  via the collocation scheme:
\begin{equation}\label{kkschm}
\begin{cases}
  \dot{p}_{N}^k(t_j^k)+cq_{N}^k(t_j^k)=d\displaystyle \sum\limits_{l=1}^{k}\displaystyle\int_{I_l}e_{\alpha,\alpha}(-\lambda
  (t_j^k-s)^\alpha) q_{N}^l(s)ds, 
 \quad 1\le j\le N; \\[9pt]
 \dot{q}_{N}^k (t_j^k)=p_{N}^k(t_j^k), \quad 1\le j\le N; \\[7pt]
 q_{N}^k(t_{k-1})=q_{N}^{k-1}(t_{k-1}),\quad   p_{N}^k(t_{k-1})=p_{N}^{k-1}(t_{k-1}).
\end{cases}
\end{equation}
\end{itemize}
At this point, some important issues need to be addressed.
 \begin{itemize}
 \item[(i)] It is known that  the solution of \eqref{eq:1T} (or \eqref{pqsystem}) has a singular behaviour at $t=0.$  We therefore subtract  $p_*,q_*$ from $p,q,$ so that $p-p_*,q-q_*$ have  higher regularity, leading to  globally higher order accuracy.  We show below that  $p_*,q_*$ can be determined analytically by following the argument in \cite{Brunner04, CaoHX03}. \medskip
 \item[(ii)] How to accurately compute the integrals involving  the singular kernel $e_{\alpha,\alpha}(\cdot)$?
 \end{itemize}
In what follows, we shall resolve these issues (see Subsections \ref{pqdefn}-\ref{wellmat}).
 
 To fix the idea, we restrict our attentions to the Chebyshev approximation.
 Let $T_n(x)=\cos(n\,{\rm arccos} x)$ be the Chebyshev polynomial of degree $n$, 
 and denote the scaled Chebyshev polynomial by
\begin{equation}\label{mappedCheby}
T_n^k(t)=T_n(x),\quad x=\frac{t-t_{k-1}}{t_k-t_{k-1}}+ \frac{t-t_{k}}{t_k-t_{k-1}},\quad t\in I_k.
\end{equation}
Hereafter, $\{x_j\}_{j=0}^N$ are the Chebyshev-Gauss-Lobatto (CGL) points.

\subsubsection{Ansatz and the formulation of $p_*,q_*$}\label{pqdefn}
%
Our starting point is to reformulate \eqref{eq:1T} into the following integral form.  This allows us to justify the well-posedness of the problem
and derive the desired $p_*, q_*$ that can capture the leading singularities.
\begin{lemma}\label{Integform} Letting $z(t)=u''(t)$,   we can rewrite \eqref{eq:1T} as
\begin{equation}\label{eq:z}
z(t)=\int_0^t \big\{de_{\alpha,\alpha+2}(-\lambda (t-s)^\alpha)-c(t-s)\big\}z(s)ds+ f(t),\ \
 \end{equation}
where
 \begin{equation}\label{ftform}
 f(t)=du_0e_{\alpha,\alpha+1}(-\lambda t^\alpha)+du_1e_{\alpha,\alpha+2}(-\lambda t^\alpha)-cu_1t-cu_0.
\end{equation}
Then  the problem  \eqref{eq:1T}  has a unique solution $u\in C(\Lambda)$.
\end{lemma}
\begin{proof}
 Solving $u''(t)=z(t)$ with  $u(0)=u_0$ and  $u'(0)=u_1,$ we find
 $$u(t)=u_0+u_1t+\int_0^t (t-s)z(s)ds.$$
 Therefore, we can rewrite \eqref{eq:1T} as
 \begin{equation}\label{eq:z0}
 z(t)+c\Big(u_0+u_1t+\int_0^t (t-s)z(s)ds\Big)=d\int_0^t e_{\alpha,\alpha}(-\lambda (t-s)^\alpha) \Big(u_0+u_1s+\int_0^s (s-\theta)z(\theta)d\theta\Big)ds.
 \end{equation}
Using the identity (cf. \cite{Pod99}): for $t>a, \alpha,\beta>0$ and $r>-1,$
\begin{equation} \label{eq:PodId}
\int_a^t e_{\rho,\gamma}(-z(t-s)^\rho)(s-a)^rds=\Gamma(r+1)e_{\rho,\beta+r+1}(-z(t-a)^\alpha),
\end{equation}
one verifies readily that
 \begin{align}
&\int_0^t e_{\alpha,\alpha}(-\lambda (t-s)^\alpha)ds
 =e_{\alpha,\alpha+1}(-\lambda t^\alpha), \quad
 \int_0^t s\,e_{\alpha,\alpha}(-\lambda (t-s)^\alpha)ds
 =e_{\alpha,\alpha+2}(-\lambda t^\alpha),\label{eq:z1}
 \end{align}
 and also by the definition \eqref{eq:EF2},
 \begin{equation}\label{eq:z2}
 \begin{split}
& \int_0^t\int_0^s e_{\alpha,\alpha}(-\lambda (t-s)^\alpha) (s-\theta)z(\theta)d\theta ds=\int_0^t\int_\theta^t e_{\alpha,\alpha}(-\lambda (t-s)^\alpha)(s-\theta)ds z(\theta)d\theta\\
&\qquad=\int_0^t \bigg\{\sum\limits_{k=0}^\infty \frac{(-\lambda)^k}{\Gamma(k\alpha+\alpha)}\int_\theta^t (t-s)^{\alpha-1+k\alpha}(s-\theta)ds\bigg\}z(\theta)d\theta\\
&\qquad=\int_0^t\sum\limits_{k=0}^\infty \frac{(-\lambda)^k (t-\theta)^{\alpha+1+k\alpha}}{\Gamma(k\alpha+\alpha+2)}z(\theta)\,d\theta
=\int_0^t e_{\alpha,\alpha+2}(-\lambda (t-\theta)^\alpha)z(\theta)\,d\theta.
\end{split}
  \end{equation}
Substituting \eqref{eq:z1}-\eqref{eq:z2}  into \eqref{eq:z0} leads to \eqref{eq:z}-\eqref{ftform}.

   Note that the operator
   $${\mathscr T}_\af [z]:=\int_0^t \big\{de_{\alpha,\alpha+2}(-\lambda (t-s)^\alpha)-c(t-s)\big\}z(s)ds$$
    is continuous, so it  is a Hilbert-Schimit operator.  It  also implies  ${\mathscr T}_\af$ is compact from $C(\Lambda)$ to $C(\Lambda)$
    \cite[p. 277]{Yosida80}.  The existence and uniqueness of the solution to \eqref{eq:z} immediately follows from the Fredholm Alternative.
\end{proof}

It is important to point out that  Brunner
(cf. \cite[Thm 6.1.6]{Brunner04}) studied  a class of integral equations with the weakly singular kernel $(t-s)^{-\mu} K(s,t),$ where $0<\mu<1$ and
$K$ is smooth, and formally characterised the singular behaviour of the solutions.  Although the result therein cannot be directly applied  to  \eqref{eq:z},
we can use the formulation of the singularity as an {\em ansatz}  to extract the most singular part of the solution of  \eqref{eq:z}.
\begin{theorem}\label{prop:gamij} For small $t>0,$
the solution of $\eqref{eq:1T}$ has the form
\begin{equation}\label{eq:ansatz}
u(t)=\sum\limits_{i,j \atop
i+j\alpha\geq2}\gamma_{ij}t^{i+j\alpha}+u_1t+u_0,
\end{equation}
where $\{\gamma_{ij}\}$ are real coefficients.
Here, the first several most singular terms of $u(t)$ can be worked out as follows
\begin{align}
u(t)&=u_*(t)+\phi(t)\notag\\
&:=\sum\limits_j\bigg\{du_1\frac{(-\lambda)^{j-3/\alpha-1}}{\Gamma(j\alpha+1)}\mathbbm{1}_{\{3/\alpha\in\mathbb{N}, 4/\alpha>j>3/\alpha\}}+du_0\frac{(-\lambda)^{j-2/\alpha-1}}{\Gamma(j\alpha+1)}\mathbbm{1}_{\{2/\alpha\in\mathbb{N}, 4/\alpha>j>2/\alpha\}}\bigg\} t^{j\alpha}\notag\\
&+\sum\limits_j\bigg\{du_1\frac{(-\lambda)^{j-2/\alpha-1}}{\Gamma(j\alpha+2)}\mathbbm{1}_{\{2/\alpha\in\mathbb{N}, 3/\alpha>j>2/\alpha\}}+du_0\frac{(-\lambda)^{j-1/\alpha-1}}{\Gamma(j\alpha+2)}\mathbbm{1}_{\{1/\alpha\in\mathbb{N}, 3/\alpha>j>1/\alpha\}}\bigg\}t^{1+j\alpha}\notag\\
&+\sum\limits_{0<j<2/\alpha}\bigg\{du_0\frac{(-\lambda)^{j-1}}{\Gamma(j\alpha+3)}+du_1\frac{(-\lambda)^{j-1/\alpha-1}}{\Gamma(j\alpha+3)}\mathbbm{1}_{\{1/\alpha\in\mathbb{N},\; j>1/\alpha\}}\bigg\}t^{2+j\alpha}\notag\\
&+\sum\limits_{0<j<1/\alpha}\bigg\{du_1\frac{(-\lambda)^{j-1}}{\Gamma(j\alpha+4)}+du_0\frac{(-\lambda)^{j+1/\alpha-1}}{\Gamma(j\alpha+4)}\mathbbm{1}_{\{1/\alpha\in\mathbb{N}\}}\bigg\}t^{3+j\alpha}+\phi(t), \label{udecomp}
\end{align}
where $\mathbbm{1}_S$ is the indicator function of the set $S,$  $\phi(t)\in C^4(\Lambda)$ and
$u_1,u_0, d$ are the same as in \eqref{eq:1T}.  With this, we take $q_*, p_*$ in  \eqref{approx1} to be
\begin{equation}\label{pqstar}
q_*(t)=u_*(t), \quad p_*(t)=u_{*}'(t).
\end{equation}
\end{theorem}
\begin{proof}
Suppose that there exists a term of the form $t^\theta, \theta<2$ in the ansatz. Substituting the term into \eqref{eq:1T} and letting $t$ approach $0$,
one easily concludes that the left hand side of \eqref{eq:1T} blows up, contradicting the right hand side, which is $0$.
 As a result, non-integer powers of the form $t^\theta, \theta<2$ are expelled in the ansatz of $u(t)$.

 On the other hand, it is impossible for us to extract  the explicit expression of $\gamma_{ij}$ for all $t^{i+j\alpha}$ as
    it is extremely tedious and complicated.
Hence, we can restrict our attention to exploiting the coefficients $\gamma_{ij}$ of term $t^{i+j\alpha}, 2<i+j\alpha<4$. 


Substituting \eqref{eq:ansatz} into \eqref{pqsystem} and using  \eqref{eq:PodId},  yield
\begin{equation}\label{eq:im}
\begin{split}
&\sum\limits_{i,j \atop i+j\alpha\geq2}\gamma_{ij}(i+j\alpha)(i-1+j\alpha)t^{i-2+j\alpha}+c\Big\{\sum\limits_{i,j \atop i+j\alpha\geq 2}\gamma_{ij}t^{i+j\alpha}+u_1t+u_0\Big\} \\
=&du_1\sum\limits_{k=0}^\infty\frac{(-\lambda)^k}{\Gamma(k\alpha+\alpha+2)}t^{(k+1)\alpha+1}+du_0\sum\limits_{k=0}^\infty\frac{(-\lambda)^k}{\Gamma(k\alpha+\alpha+1)}t^{(k+1)\alpha}\\
&+d\sum\limits_{i,j\atop i+j\alpha> 2}\Gamma(i+1+j\alpha)\gamma_{ij}\sum\limits_{k=0}^\infty \frac{(-\lambda)^k}{\Gamma(k\alpha+\alpha+i+1+j\alpha)}t^{(k+1+j)\alpha+i}.
\end{split}
\end{equation}

Now, we equate powers of lower order terms $t^{1+j\alpha}, t^{j\alpha}, t^{j\alpha-1}$ and $t^{j\alpha-2}$ for the following four cases respectively.
It is noteworthy to point out that monomials are excluded out of our consideration for these cases.

\vskip 3pt
{\bf Case 1}: $\big\{j\,: \    3+j\alpha<4,\ j\in\mathbb{N}\big\}$
\vskip 3pt

We consider similar terms of the form $t^{1+j\alpha}$. Note that the candidates in the right hand side of \eqref{eq:im} which could have the form are $t^{(k+1)\alpha+1}$ and  $t^{(k+1)\alpha}$. Let
\begin{equation}
\begin{cases}
1+j\alpha=(k+1)\alpha+1 \quad  \Rightarrow\quad k=j-1,\notag\\
1+j\alpha=(k+1)\alpha   \qquad\ \ \Rightarrow\quad k=j-1+1/\alpha,\ \ \text{if}\ 1/\alpha\in\mathbb{N}.
\end{cases}
\end{equation}
Hence, equating coefficients of $t^{1+j\alpha}$ on both sides of \eqref{eq:im} yields
 \begin{align}\label{eq:gam3j}
\gamma_{3j}(3+j\alpha)(2+j\alpha)=du_1\frac{(-\lambda)^{j-1}}{\Gamma(j\alpha+2)}+du_0\frac{(-\lambda)^{j+1/\alpha-1}}{\Gamma(j\alpha+2)}\mathbbm{1}_{\{1/\alpha\in\mathbb{N}\}},\notag\\
\quad\qquad\qquad\quad \qquad\gamma_{3j}=du_1\frac{(-\lambda)^{j-1}}{\Gamma(j\alpha+4)}+du_0\frac{(-\lambda)^{j+1/\alpha-1}}{\Gamma(j\alpha+4)}\mathbbm{1}_{\{1/\alpha\in\mathbb{N}\}}.
\end{align}

\vskip 3pt
{\bf Case 2}: $\big\{j \ :\   2+j\alpha<4, \  j\in\mathbb{N} \big\}$
\vskip 3pt

Now, we consider similar terms of the form $t^{j\alpha}$. Similar as the previous case by considering two candidates $t^{(k+1)\alpha+1}$ and  $t^{(k+1)\alpha}$ of the right hand side of \eqref{eq:im}, we have
\begin{equation}
\begin{cases}
j\alpha=(k+1)\alpha \qquad\ \   \Rightarrow\quad k=j-1,\notag\\
j\alpha=(k+1)\alpha+1   \quad \Rightarrow\quad k=j-1-1/\alpha,\ \ \text{if}\ 1/\alpha\in\mathbb{N} \ \text{and}\ j>1/\alpha.
\end{cases}
\end{equation}
Equating coefficients for $t^{j\alpha}$ on both sides of \eqref{eq:im} implies
 \begin{align}\label{eq:gam2j}
\gamma_{2j}(2+j\alpha)(1+j\alpha)=du_0\frac{(-\lambda)^{j-1}}{\Gamma(j\alpha+1)}+du_1\frac{(-\lambda)^{j-1/\alpha-1}}{\Gamma(j\alpha+1)}\mathbbm{1}_{\{1/\alpha\in\mathbb{N}, j>1/\alpha\}},\notag\\
\quad\qquad\qquad\quad \qquad\gamma_{2j}=du_0\frac{(-\lambda)^{j-1}}{\Gamma(j\alpha+3)}+du_1\frac{(-\lambda)^{j-1/\alpha-1}}{\Gamma(j\alpha+3)}\mathbbm{1}_{\{1/\alpha\in\mathbb{N}, j>1/\alpha\}}.
\end{align}

{\bf Case 3}: $\big\{j \ :\  1+j\alpha<4, \   j\in\mathbb{N}\big\}$
\vskip 3pt

For the term $t^{j\alpha-1}$, we follow the same fashion to have
 \begin{equation}
\begin{cases}
j\alpha-1=(k+1)\alpha+1 \qquad\Rightarrow\quad k=j-2/\alpha-1,\ \ \text{if}\ 2/\alpha\in\mathbb{N} \ \text{and}\ j>2/\alpha\notag\\
j\alpha-1=(k+1)\alpha  \quad \qquad\ \  \Rightarrow\quad k=j-1/\alpha-1,\ \ \text{if}\ 1/\alpha\in\mathbb{N} \ \text{and}\ j>1/\alpha.
\end{cases}
\end{equation}
Equating coefficients for $t^{j\alpha-1}$ yields
 \begin{align}\label{eq:gam1j}
\gamma_{1j}(1+j\alpha)(j\alpha)&=du_1\frac{(-\lambda)^{j-2/\alpha-1}}{\Gamma(j\alpha)}\mathbbm{1}_{\{2/\alpha\in\mathbb{N}, 3/\alpha>j>2/\alpha\}}+du_0\frac{(-\lambda)^{j-1/\alpha-1}}{\Gamma(j\alpha)}\mathbbm{1}_{\{1/\alpha\in\mathbb{N}, 3/\alpha> j>1/\alpha\}},\notag\\
\gamma_{1j}&=du_1\frac{(-\lambda)^{j-2/\alpha-1}}{\Gamma(j\alpha+2)}\mathbbm{1}_{\{2/\alpha\in\mathbb{N}, 3/\alpha>j>2/\alpha\}}+du_0\frac{(-\lambda)^{j-1/\alpha-1}}{\Gamma(j\alpha+2)}\mathbbm{1}_{\{1/\alpha\in\mathbb{N}, 3/\alpha>j>1/\alpha\}}.
\end{align}

{\bf Case 4}: $\big\{j\ :\    j\alpha<4,\  j\in\mathbb{N}\big\}$
\vskip 3pt 

Finally, we consider the term $t^{j\alpha-2}$,
 \begin{equation}
\begin{cases}
j\alpha-2=(k+1)\alpha+1 \qquad\Rightarrow\quad k=j-3/\alpha-1,\ \ \text{if}\ 3/\alpha\in\mathbb{N} \ \text{and}\ j>3/\alpha,\notag\\
j\alpha-2=(k+1)\alpha  \quad \qquad\ \  \Rightarrow\quad k=j-2/\alpha-1,\ \ \text{if}\ 2/\alpha\in\mathbb{N} \ \text{and}\ j>2/\alpha.
\end{cases}
\end{equation}
Equating coefficients for $t^{j\alpha-2}$ leads to
 \begin{align}\label{eq:gam0j}
& \gamma_{0j}(j\alpha)(j\alpha-1)\notag\\
&=du_1\frac{(-\lambda)^{j-3/\alpha-1}}{\Gamma(j\alpha-1)}\mathbbm{1}_{\{3/\alpha\in\mathbb{N}, 4/\alpha> j>3/\alpha\}}+du_0\frac{(-\lambda)^{j-2\alpha-1}}{\Gamma(j\alpha-1)}\mathbbm{1}_{\{2/\alpha\in\mathbb{N}, 4/\alpha>j>2/\alpha\}},\notag\\
\gamma_{1j}&=du_1\frac{(-\lambda)^{j-3/\alpha-1}}{\Gamma(j\alpha+1)}\mathbbm{1}_{\{3/\alpha\in\mathbb{N}, 4/\alpha>j>3/\alpha\}}+du_0\frac{(-\lambda)^{j-2/\alpha-1}}{\Gamma(j\alpha+1)}\mathbbm{1}_{\{2/\alpha\in\mathbb{N}, 4/\alpha>j>2/\alpha\}}.\notag
\end{align}

Once  lower order terms (i.e., $t^{i+j\alpha}$ with $i+j\alpha<4$) are determined, the remainder is wrapped up into $\phi(t)\in C^4(\Lambda)$.
\end{proof}

\begin{rem}{\em 
We exclude the cases $i+j\alpha\in \mathbb{N}$ in that polynomials can be absorbed into $\phi(t)$. }
\end{rem}

\subsubsection{Mapping  techniques for evaluating weakly singular integrals}   In the implementation of the scheme  \eqref{k1schm}-\eqref{kkschm},
we have to deal with singular integrals of Type-I:
\begin{equation}\label{type1sing}
\begin{split}
 {\mathcal I}_\alpha^{\rm I}(t) =&\int_{t_{k-1}}^{t}   e_{\alpha,\alpha}(-\lambda (t-s)^\alpha) g(s)\, ds,\\
& {\rm for}\;\;  g(s)=s^\beta,\  t=t_j^1\in (t_0,t_1], \;\; k=1,  
\\& \text{or}\;\;    g(s)=T_n^k(s), \ t=t_j^k\in (t_{k-1},t_k],\;k=1,2,\cdots, K,
\end{split}
\end{equation}
and the nearly singular integers of  Type-II:
\begin{equation}\label{type2sing}
\begin{split}
 {\mathcal I}_\alpha^{\rm I\!I}(t)=& \int_{t_{k-1}}^{t_k}   e_{\alpha,\alpha}(-\lambda (t-s)^\alpha) g(s)\, ds, \\
& {\rm for}\;\;   g(s)=s^\beta,\ t>t_1,\; t\approx t_1,\;  k=1;\\
& {\rm or}\;\; g(s)=T_n^k(s), t>t_k,\; t\approx t_k;  \;\; k=1,2,\cdots, K,
\end{split}
\end{equation}
where $\beta\in {\mathbb R}$ relates to the aforementioned ansatz $p_*,q_*$ in the first subinterval $[0,t_1].$

The difficulty of approximating  both types resides in the fact that the kernel $e_{\alpha,\alpha}(\cdot)$ has  infinitely many terms of singular powers with different singular behaviours  (cf. \eqref{MLfuncA} and \eqref{eafaf}).  As a result,  a numerical quadrature, e.g. Jacobi-Gauss quadrature, involving a single weight function  cannot provide  the satisfactory accuracy.   Indeed,
 we depict in Figure \ref{Fig:illu}  the integrands with several parameters, and observe that  the integrands exhibit heavy boundary layers at one end of the interval.
\begin{figure}[!htb]
\centering
\resizebox{140mm}{70mm}{\includegraphics{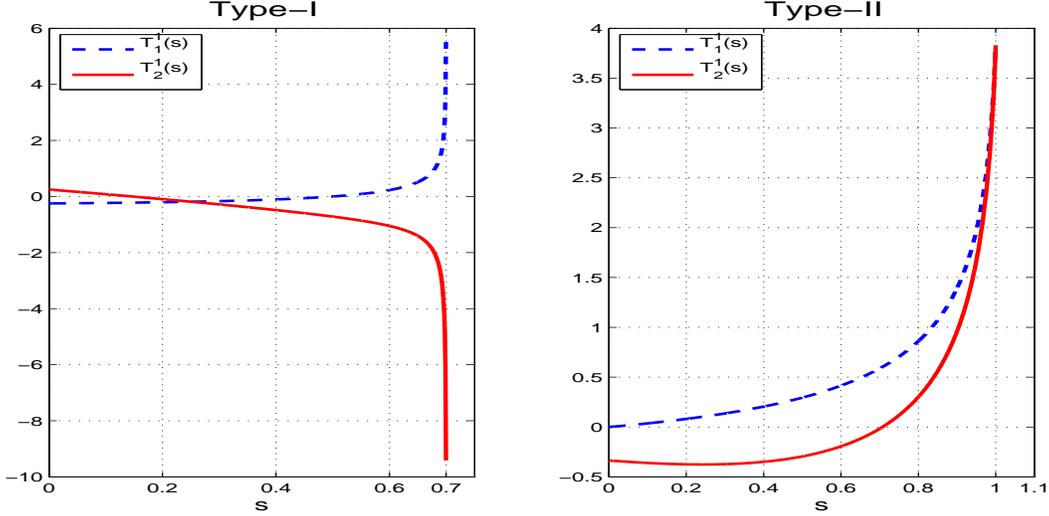}}
\caption{(Left): A plot of $e_{0.6,0.6}(-(0.7-s)^{0.6})T_n^1(s), n=1\ \text{or}\ 2, s\in [0,0.7]$; (Right): A plot of $e_{0.6,0.6}(-(1.01-s)^{0.6})T_n^1(s), n=1\ \text{or}\ 2, s\in [0,1]$.} \label{Fig:illu}
\end{figure}


To surmount this obstacle, we resort to  the mapping technique that can redistribute the quadrature points to the end of the interval where they are
mostly needed to resolve the boundary layer.   Following the idea of \cite{WangS05}, we introduce the one-sided singular mapping:
\begin{equation}
t=h(y;r)=t_r+(t_l-t_r) \Big(\frac{1-y} 2\Big)^{1+r},  \;\;\; y\in [-1,1],\;\; t\in [t_l,t_r], \;\;\;   r\in\mathbb{N}.
\end{equation}
Let $\{y_i,\omega_i\}_{i=0}^N$ be the  Gauss-Legendre quadrature points and weights on $[-1,1],$ and define the mapped points $\{t_i=h(y_i;r)\}_{i=0}^N.$ Denote  by
$f(t)$  a generic integrand on $(t_l,t_r)$ with a singular layer near $t=t_r.$  Basically,  we have
\begin{equation}\label{singint}
\int_{t_l}^{t_r} f(t)dt=c_r\int_{-1}^1 f(h(y;r))  \Big(\frac{1-y} 2\Big)^{r} dy\approx c_r\sum_{i=0}^N f(t_i) \Big(\frac{1-y_i} 2\Big)^{r}\omega_i,
\end{equation}
where $c_r=(r+1)\frac{t_r-t_l} 2.$ We see that with the  factor $(1-y)^r$, the integrand is much better behaved  in $y.$  On the other hand,  more and more points
are clustered near $t=t_r$ as $r$ increases.
To demonstrate the gain of the mapping technique,  we consider  two examples of different type:  (i) $f(t)=e_{0.6,0.6}(-(0.7-t)^{0.6})T_n^1(t), t\in (0,0.7),$ and
(ii) $f(t)=e_{0.6,0.6}(-(1.01-t)^{0.6})T_n^1(t), t\in (0,1).$  Note that we can calculate  the exact values of two integrals by using the property of ML-functions.

In Figure \ref{Fig:illu2},  we depict  the error curves of the usual quadrature and  the mapped approaches (i.e., $r=0$ and $r=3$) against various $N.$
We observe a much faster decay of the errors from the mapped approach.   Therefore, with the mapping, we can compute the singular/nearly singular
integrals much more accurately.


\begin{figure}[!htb]
\centering
\resizebox{130mm}{70mm}{\includegraphics{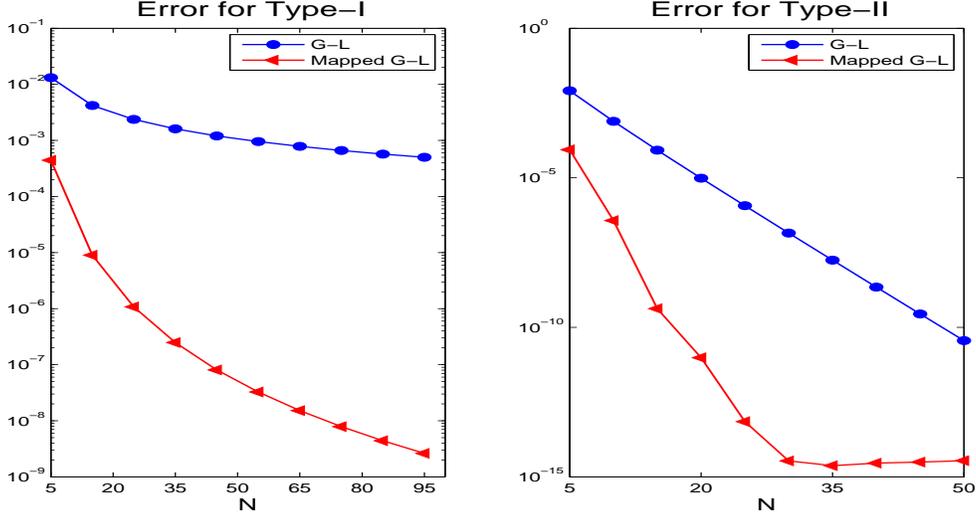}}
\caption{Errors of  Gauss-Legendre quadrature (G-L) and mapped G-L quadrature (with $r=3$). Left: Case (i); Right: Case (ii).} \label{Fig:illu2}
\end{figure}

%
%
%
%

\subsubsection{Well-conditioned collocation matrix}\label{wellmat}
The third issue of marching collocation scheme is that the condition number of standard collocation matrix  ${D}$ associated with the second order term $u_{tt}$ grows like $\mathcal{O}(N^4)$, where $N$ is the number of collocation points. To circumvent the difficulty, we first rewrite \eqref{eq:1T} into a damped Hamiltonian system  with only first order derivatives and then construct the explicit inverse matrix $B$ for first order collocation matrix through Birkhoff interpolation.

Here, we only list the explicit form of $B=B_j(x_i), 1\leq i,j\leq N$. On the standard interval $[-1,1]$, given  Chebyshev  collocation points and its associated weights $\{x_i,w_i\}_{i=0}^N$ with increasing order, $B_j(x)$ has the following form.
 \begin{align}
& B_j(x)=\sum\limits_{k=0}^{N-1} w_j[T_k(x_j)-T_N(x_j)(-1)^{N+k}] \partial^{-1}_xT_k(x),
\end{align}
where $\partial_x^{-1} T_k(x) = \int_{-1}^x T_k(y)dy,$ and 
\begin{align}
& \partial^{-1}_x T_0(x)=1+x, \;\;\;  \partial^{-1}_x T_1(x)=\frac{x^2-1}{2},\notag\\
& \partial^{-1}_x T_k(x)=\frac{T_{k+1}(x)}{2(k+1)}-\frac{T_{k-1}(x)}{2(k-1)}-\frac{(-1)^k}{k^2-1}, \;\;\;  k\geq 2.
\end{align}
The readers are referred to \cite{WangSZ14} for the details, where   the computation of $B$
is stable even for thousands of collocation points.

\subsection{Numerical experiments}
\begin{example}
Consider the  equation
 \begin{equation}\label{eq:exode}
 \begin{cases}
u''(t)+4u(t)=3\displaystyle\int_0^t e_{\alpha,\alpha}(-1.5(t-s)^{0.6})u(s)ds,\;\;  t\in[0,20],\\[8pt]
  u(0)=0, \ u'(0)=2.
 \end{cases}
 \end{equation}
\end{example}
We partition the domain into $20$ equidistant subintervals. Since the solution is singular near $t=0$, we take advantage of the ansatz \eqref{udecomp} for the first subinterval and use the approximation \eqref{k1schm}.  For other intervals, we apply standard polynomial approximation \eqref{kkschm}. Clearly, we can define the Hamiltonian $H(t)=p^2(t)+2q^2(t)$.

\begin{figure}[!htb]
\centering
\resizebox{130mm}{65mm}{\includegraphics{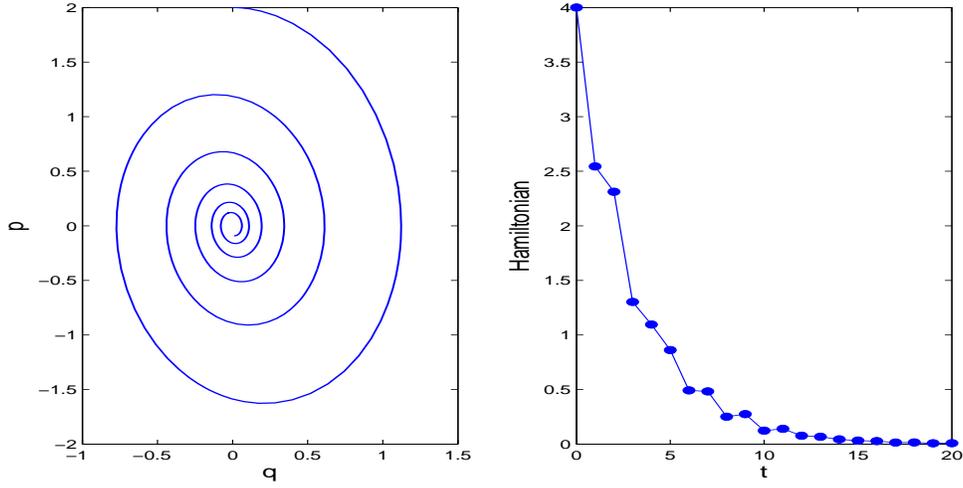}}
\caption{(Left): A phase plot for \eqref{eq:exode} with $u(0)=0, u_t(0)=2$ by using $20$ collocation points on each time interval; (Right): A plot of Hamiltonian decay with respect to time.} \label{Fig:ExODE}
\end{figure}
Indeed, as we  observe from Figure \ref{Fig:ExODE}, the Hamiltonian decreases as time increases. The system stays at the origin when it reaches the steady state.

To validate the necessity of condition $u_0=0$ in Theorem \ref{thm:1D}, we switch the initial condition of \eqref{eq:exode} to $u(0)=2, u_t(0)=0$ and obtain  Figure \ref{Fig:ExODE1}. One can easily observe that as time proceeds, the Hamiltonian  may exceed the initial one, which contracts Theorem \ref{thm:1D}.
 \begin{figure}[!htb]
\centering
\resizebox{130mm}{65mm}{\includegraphics{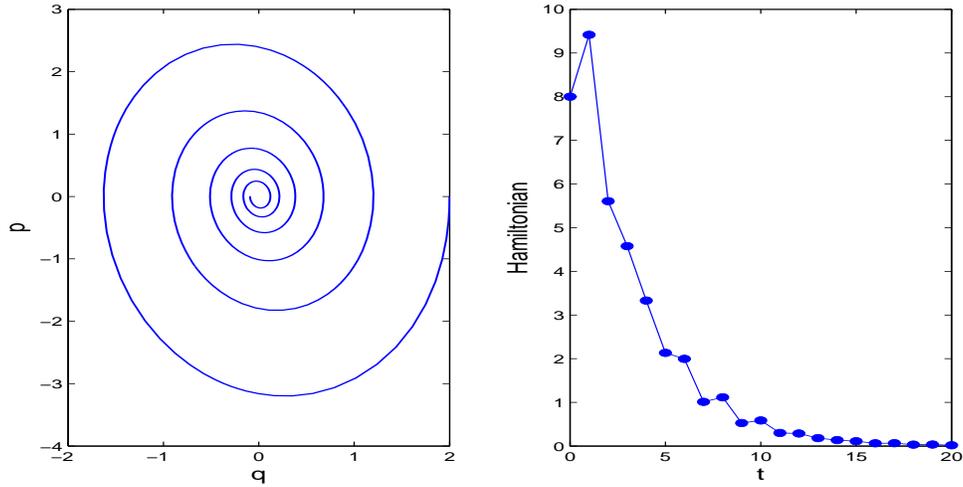}}
\caption{(Left): A phase plot for \eqref{eq:exode} with $u(0)=2, u_t(0)=0$ by using $20$ collocation points on each time interval; (Right): A plot of Hamiltonian decay with respect to time. Note that under this initial condition, the decay is not strict. }\label{Fig:ExODE1}
\end{figure}

\begin{example}
To validate the special treatment of our algorithm in the first subinterval, we consider the equation
\begin{equation}
\begin{cases}
\ddot{u}(t)+cu(t)=d\displaystyle \int_0^t e_{\alpha,\alpha}(-\lambda (t-s)^\alpha)u(s)ds+g(t),\\[8pt]
u(0)=u_0,\quad \dot u(0)=u_1,
\end{cases}
\end{equation}
where initial conditions and source term $g(t)$ are chosen such that
\begin{equation}
u(t)=t^{2+\alpha}+t^{3+2\alpha}+ \begin{cases}
\; \;(t-1)^5,\ \ t\leq 1,\\
-(t-1)^5,\ \ t>1.
\end{cases}
\end{equation}
\end{example}
Here, we aim to mimic  the ansatz in  Proposition \ref{prop:gamij}. From our algorithm, $\tau=2$ implies direct polynomial approximation for $u(t)$,
$\tau=3$ leads to  polynomial approximation for the last two terms of $u(t)$, and $\tau=5$ indicates polynomial approximation for the last term.
Numerical results are shown in the Figure \ref{Fig:ExODE2}. The number in the parentheses means the slope of associated reference line.
 \begin{figure}[!htb]\label{Fig:ExODE2}
\centering
\resizebox{120mm}{70mm}{\includegraphics{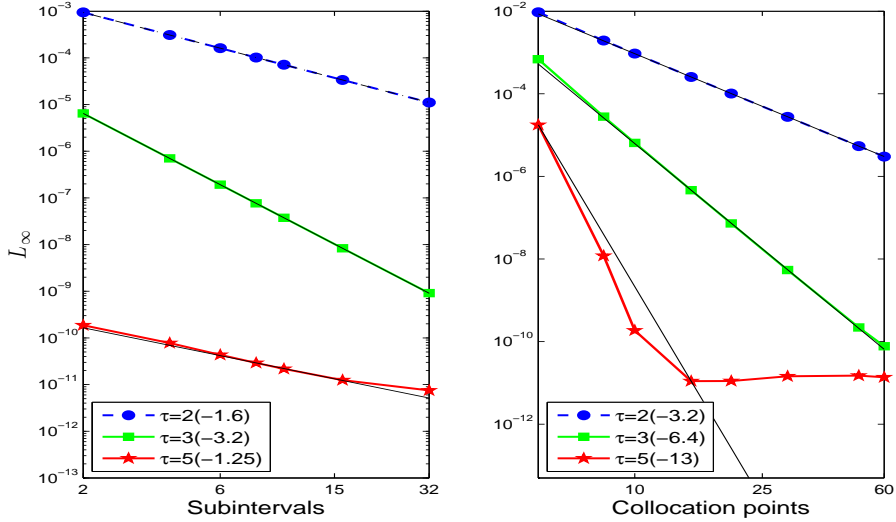}}
\caption{(Left): Numerical error for interval refinement with frozen 9 collocation points; (Right): Numerical error of approximation with $2$ equal-length subintervals, on which various collocation points are applied. }
\end{figure}

\vskip .1in

\subsection{Error analysis} To begin with, we present an important result on Chebyshev-Gauss-Lobatto  interpolation  of the singular   function: $h(t)=(t+1)^\theta, t\in [-1,1]$ and real $\theta>0$.
 \begin{lemma}\label{lem:interpolant}
Let  $I_N$ be the interpolation operator on the Chebyshev-Lobatto points $\{t_i\}_{i=0}^N$. Then
\begin{equation}
\|h-I_Nh\|_\infty\leq 2N^{-2\theta}.
\end{equation}
\end{lemma}
\begin{proof}
Let $a_n$ denote the exact Chebyshev  expansion coefficient of $h(t)$, i.e. $a_n=\int_{-1}^1 h(t)T_n(t)w(t)dt$, where $w(t)=(1-t^2)^{-1/2}$.
Then, a careful computation (cf. \cite[Lemma 4]{HuangS16} implies
\begin{equation}
a_n=\mathcal{O}(n^{-1-2\theta}).
\end{equation}
Furthermore, denote
$I_Nh(t)=\sum\limits_{n=0}^N\!\!{''}\, b_nT_n(x)$, where double prime means the first and the last terms are to be taken by a factor of $1/2$.
Apply \cite[Theorem 21]{Boyd01} to get
\begin{equation}
\|u-I_Nu\|_\infty\leq 2\sum\limits_{n=N+1}^\infty |a_n|=2N^{-2\theta}.
\end{equation}
This ends the proof. 
\end{proof}
\begin{rem} {\em
We note that  \cite[Theorem 3.40]{ShenTW11} provides a convergence rate for Chebyshev interpolation for rather general functions.
However, by taking advantage of the concrete form of $h(t)$,  we can get significantly better convergence rate. }
\end{rem}

For the sake of analysis,  we define the operators $\mathcal{A}^j, \mathcal{B}^j : C(I_j)\to C(I_j)$ on each $I_j$  by
\begin{equation}
(\mathcal{A}^ju)(t)=\int_{I_j} e_{\alpha,\alpha}(-\lambda (t-s)^\alpha)u(s)ds,\ \ t>t_j,
\end{equation}
and
\begin{equation}
(\mathcal{B}^ju)(t)=\int_{t_{j-1}}^t e_{\alpha,\alpha}(-\lambda (t-s)^\alpha)u(s)ds,\ \ t\in I_j.
\end{equation}
Then, there exists a best polynomial $\pi_N(\mathcal{B}^ju)$ of order $N$ such that (cf. \cite[Lemma 7]{HuangS16})
\begin{equation}
\|\mathcal{K}^ju-\pi_N(\mathcal{B}^ju)\|_\infty \leq CN^{-\alpha} \|u\|_\infty.   
\end{equation}

\begin{theorem}
Assume the ansatz \eqref{eq:ansatz} for $u$ when $t\to 0$ and $u\in H^m(0,T]$ for some $m>5/2$, and $c-d/\lambda\geq 0, \tau=4$.
 Then, for our marching scheme on the whole time span $[0,T]$, there holds
\begin{equation}\label{eq:cr}
\|p-p_N\|_{\infty}+(c-d/\lambda)\|q-q_N\|_\infty\leq CN^{-\min\{4,m-5/2\}},
\end{equation}
where $C$ depends on $u, T$ but is independent of $N$.
\end{theorem}
\begin{proof}
 Define the error function on $I_j$ by
$$e_{p,j}(t)=p^j(t)-p_N^j(t), e_{q,j}(t)=q^j(t)-q_N^j(t), \ t\in I_j. $$
Recall from \eqref{udecomp} that on the interval $I_1$, we denote
\begin{align}\label{eq:thmpq}
& {q}^1(t)=q_*(t)+\phi(t), \quad  p^1(t)=p_*(t)+\phi'(t). 
\end{align}
 Hence,  we  have
 $$e_{p,1}(t)={p}^1(t)-p_N^1(t)=\phi'(t)-\hat{p}_N^1(t), e_{q,1}(t)=q^1(t)-q_N^1(t)=\phi(t)-\hat{q}_N^1(t).$$
Then, on each $I_j$, substituting \eqref{k1schm} or \eqref{kkschm} into \eqref{pqsystem}, and subtracting the resulted equation from \eqref{pqsystem}, we have
\begin{equation}\label{eq:thmpq1}
\begin{cases}
\dot{{p}}^j(\xi_i)-\dot{p}_N^j(\xi_i) =-c{q}^j(\xi_i)+cq_N^j(\xi_i)+d\displaystyle \sum\limits_{k=1}^{j-1}\int_{I_k}e_{\alpha,\alpha}(-\lambda(\xi_i-s)^\alpha) e_{q,k}(s)ds\\
 \quad\qquad\qquad\qquad +  d \displaystyle \int_{t_{j-1}}^{\xi_i} e_{\alpha,\alpha}(-\lambda (\xi_i-s)^\alpha) e_{q,j}(s)ds,\\[8pt]
 \dot{{q}}^j(\xi_i)-\dot{q}_N^j(\xi_i)={p}^j(\xi_i)-p_N^j(\xi_i),\\
e_{p,j}(t_{j-1})=e_{p,j-1}(t_{j-1}), e_{q,j}(t_{j-1})=e_{q,j-1}(t_{j-1}),
\end{cases}
\end{equation}
where $\{\xi_i\}_{i=0}^N$ is the Chebyshev-Lobatto collocation points on $I_j$.
Multiply both sides of \eqref{eq:thmpq1} by $l_i(t)$ and sum over $i$, where $l_i(t)$ is the Lagrange interpolation basis associated with $\xi_i$ to obtain
\begin{equation}
\begin{cases}
I_N\dot{p}^j-\dot{p}_N^j=-aI_N{q}^j+aq_N^j+bI_N\sum\limits_{k=1}^{j-1}\mathcal{A}^k e_{q,k}+bI_N\mathcal{B}^je_{q,j}, \\
I_N\dot{q}^j-\dot{q}_N^j=I_N{p}^j-p_N^j.
\end{cases}
\end{equation}
Since $e_{p,j}=p^j-I_Np^j+I_Np^j-p_N^j$ and $e_{q,j}=q^j-I_Nq^j+I_Nq^j-q_N^j$,
 we therefore have the error function
 \begin{equation}\label{eq:error}
 \begin{cases}
 \dot{e}_{p,j}(t)=-ae_{q,j}(t)+b\int_{t_{j-1}}^t e_{\alpha,\alpha}(-\lambda (t-s)^\alpha)e_{q,j}(s)ds+F(t),\\
 \dot{e}_{q,j}(t)=e_{p,j}(t)+G(t),\\
 e_{p,j}(t_{j-1})=e_{p,j-1}(t_{j-1}), e_{q,j}(t_{j-1})=e_{q,j-1}(t_{j-1}),
 \end{cases}
 \end{equation}
 where
 \begin{align}
 & F(t)=\underbrace{\dot{p}^j-I_N{\dot{p}^j}}_{F_1}+\underbrace{a({q}^j-I_N{q^j})}_{F_2}+\underbrace{bI_N\sum\limits_{k=1}^{j-1}\mathcal{A}^k e_{q,k}}_{F_3}
 +\underbrace{b(I_N-I)\mathcal{B}^je_{q,j}(s)}_{F_4},\\
 &G(t)=\underbrace{\dot{q}^j-I_N{\dot{q}^j}}_{G_1}+\underbrace{I_N{p^j}-{p^j}}_{G_2}.
 \end{align}
 Integrating both sides of \eqref{eq:error} from $0$ to $\xi$ and following the proof of Theorem \ref{thm:energy}, we obtain
 \begin{align}
 e_{p,j}^2(\xi)&+(a-b/\lambda)e_{q,j}^2(\xi)\notag\\
 &\leq e_{p,j}^2(t_{j-1})+(a-b/\lambda)e_{q,j}^2(t_{j-1})-\frac{2be_{q,j}(t_{j-1})}{\lambda}\int_{t_{j-1}}^\xi \dot{e}_{q,j}(t)E_{\alpha,1}(-\lambda t^\alpha)dt\notag\\
 &\quad+\int_{t_{j-1}}^\xi F(t)e_{p,j}(t)dt+a\int_{t_{j-1}}^\xi G(t)e_{q,j}(t)dt.
 \end{align}
 Again, the second mean value theorem implies there exists a $\xi_0\in (t_{j-1},\xi)$ such that
 \begin{align}
 -\frac{2be_{q,j}(t_{j-1})}{\lambda}\int_{t_{j-1}}^\xi & \dot{e}_{q,j}(t)E_{\alpha,1}(-\lambda t^\alpha)dt= -\frac{2be_{q,j}(t_{j-1})E_{\alpha,1}(-\lambda t_{j-1}^\alpha)}{\lambda}\int_{t_{j-1}}^{\xi_0} \dot{e}_{q,j}(t)dt \notag\\
 &= \frac{2be_{q,j}(t_{j-1})E_{\alpha,1}(-\lambda t_{j-1}^\alpha)}{\lambda} (e_{q,j}(t_{j-1})-e_{q,j}(\xi_0)) \notag\\
 &\leq \bigg(\frac{2b}{\lambda}+\frac{b}{\lambda\epsilon}\bigg)e_{q,j}^2(t_{j-1})+\frac{b\epsilon}{\lambda}\|e_{q,j}\|_\infty^2,
 \end{align}
 where $\epsilon$ is an arbitrarily small positive number.
 Hence,
 \begin{align}
 e_{p,j}^2(\xi)&+(a-b/\lambda)e_{q,j}^2(\xi)
 \leq e_{p,j}^2(t_{j-1})+(a+b/\lambda+b/\lambda\epsilon)e_{q,j}^2(t_{j-1})+b\epsilon/\lambda\|e_{q,j}\|_\infty^2\notag\\
 &\quad+\frac{1}{2}\|F\|_\infty^2+\frac{1}{2}\|e_{p,j}\|_\infty^2+\frac{(a-b/\lambda)}{2}\|e_{q,j}\|_\infty^2+\frac{a^2}{2(a-b/\lambda)}\|G\|_\infty^2.
 \end{align}
 Since the inequality holds for all $\xi\in I_j$, we clearly have for $\epsilon\to 0$
  \begin{align}\label{eq:ej}
 \|e_{p,j}\|_\infty^2+(a-b/{\lambda})\|e_{q,j}\|_\infty^2
&\leq C(e_{p,j}^2(t_{j-1})+e_{q,j}^2(t_{j-1})
+\|F\|_\infty^2+\|G\|_\infty^2),
 \end{align}
 where $$C=\max\bigg\{2,2a+\frac{2b}{\lambda}+\frac{2b}{\lambda\epsilon}, \frac{a^2}{a-b/\lambda}\bigg\}.$$

 With the stability inequality at our disposal,  we next prove the convergence rate on $I_j$ by induction.

 When $j=1$, it is obvious that $e_{p,1}(0)=0=e_{q,1}(0)$.
 Next, let us bound $\|F\|_\infty$ and $\|G\|_\infty$.
 Note that in this case $F_3=0$,
Then, Lemma \ref{lem:interpolant} immediately indicates
 \begin{equation}\label{eq:FG}
 \|F_1\|_\infty=\|\phi''-I_N\phi''\|_\infty\leq CN^{-4}.
 \end{equation}
 Similarly, we have
$
  \|F_2\|_\infty\leq CN^{-8}, \ \|G_1\|_\infty\leq CN^{-6}, \ \text{and }\|G_2\|_\infty\leq CN^{-6}.
$
Moreover,
\begin{align}\label{eq:F4}
\|F_4\|_\infty&=b\|(I-I_N)\mathcal{B}e_{q,1}(s)\|_\infty,\notag\\
&\leq C\|(I-I_N)(\mathcal{B}e_{q,1}-B_N\mathcal{B}e_{q,1})\|_\infty,\notag\\
&\leq C(1+\log N)\|\mathcal{B}e_{q,1}-B_N\mathcal{B}e_{q,1}\|_\infty,\notag\\
&\leq C(1+\log N) N^{-\alpha}\|e_{q,1}\|_\infty,
\end{align}
where $\log N$ is the Lebesgue constant of the operator $I_N$.

Combining \eqref{eq:FG}--\eqref{eq:F4}, we have
\begin{equation}
 \|e_{p,1}\|_{\infty}^2+(a-b/\lambda)\|e_{q,1}\|_\infty^2\leq CN^{-8}+C(1+\log N) N^{-2\alpha}\|e_{q,1}\|_\infty^2.
 \end{equation}
 For $N$ sufficiently large, we can always have $(1+\log N)^2N^{-2\alpha}\leq (a-b/\lambda)/2C$.  Therefore,
 \begin{equation}
 \|e_{p,1}\|_{\infty}\leq CN^{-4}, \ \ \text{and} \ \ \|e_{q,1}\|_\infty \leq CN^{-4}.
 \end{equation}
Hence, \eqref{eq:cr} is true for $j=1$.

Suppose our estimate is true for all $j=1,\cdots, k$, let us consider the case $j=k+1$. From \eqref{eq:ej}, the argument is similar to the case $j=1$, except for the use  of \cite[(5.5.28)]{CanutoHQZ06}: 
\begin{align}
&\|F_1\|_\infty=\|\dot{p}^j-I_N\dot{p}^j\|_\infty \leq CN^{5/2-m},
\|F_2\|_\infty=\|q^j-I_Nq^j\|_\infty\leq CN^{1/2-m},\notag\\
&\|G_1\|_\infty=\|\dot{q}^j-I_N\dot{q}^j\|_\infty \leq CN^{3/2-m},  \|G_2\|_\infty=\|p-I_Np\|_\infty \leq CN^{3/2-m}.
\end{align}

Since $E_{\alpha,1}(-\lambda (t-s)^\alpha)$ is increasing on $s$, we conclude $e_{\alpha,\alpha}(-\lambda (t-s)^\alpha)\geq 0$. Thus,
\begin{align}
\|F_3\|_\infty&=\bigg\|\sum\limits_{n=1}^k \int_{t_{n-1}}^{t_n} e_{\alpha,\alpha}(-\lambda (t-s)^\alpha) e_{q,n}(s)ds\bigg\|_\infty \notag\\
 &\leq \max\limits_{1\leq n\leq k}\|e_{q,n}\|_\infty \int_0^{t_k} e_{\alpha,\alpha}(-\lambda (t-s)^\alpha)ds\notag\\
 &=\max\limits_{1\leq n\leq k}\|e_{q,n}\|_\infty [E_{\alpha,1}(-\lambda(t-t_n)^\alpha)-E_{\alpha,1}(-\lambda t^\alpha)]\notag\\
 &\leq \max\limits_{1\leq n\leq k}\|e_{q,n}\|_\infty \leq CN^{-\min\{4,m-5/2\}}.
\end{align}
Therefore,
  \begin{align}
 \|e_{p,k+1}\|_\infty^2+(a-{b}/{\lambda})\|e_{q,k+1}\|_\infty^2
 &\leq CN^{-\min\{8,2m-5\}},
 \end{align}
 where $C$ depends on $u, a,b,\lambda$ and $T$, but independent of $N$. This ends the proof.
\end{proof}
\begin{rem}
{\em If $u(t)$ satisfies the condition  that it has an absolutely continuous $(m-1)$st derivative $u^{(m-1)}$ on
$[0, T]$ for some $m>2$ with $u^{(m-1)}(t) = m^{(m-1)}(0) + \int_{0}^Tg(y)dy$, where $g$ is
absolutely integrable and of bounded variation $Var(g) <\infty$ on $[0,T]$, we can easily improve the result \eqref{eq:cr} to
$$\|p-p_N\|_{\infty}+(a-b/\lambda)\|q-q_N\|_\infty\leq CN^{-\min\{4,m-2\}}$$
by using \cite[Theorem 4.5]{Xiang16}. }
\end{rem}

\subsection{Numerical experiments}
\begin{example}
Consider the one-dimensional Cole-Cole model \eqref{Prob:CC} with $x\in [0,2]$. At $t=0$, we choose  initial square impulse on $x\in [0.9,1.1]$ and $u_t(x,0)=0$. 
\end{example}
To be consistent with the parameters used in numerical experiments of \cite[p. 61]{Causley11}, we take $c=1$ and $d=74/75$. Clearly, we observe that the electric field propagates $E$ evolves in a similar ashion as solution of classical wave equation in a finite interval domain (cf. \cite[p. 63]{Strauss07}), which is, a wave bounces back and force many times. Unlike the classical problem, the magnitude of $E$ in our example damps along with time because of energy loss.  The time evolution of electric field $E$ of  \eqref{eq:EF2} for $\alpha=0.6, T=1.5$ is presented in Figure \ref{Fig:Ex2}. In the experiment, we use polynomial degree of order $200$ in spatial approximation and collocation points of number $20$ on each time subinterval of length $0.3$.

\begin{figure}[!htb] \label{Fig:Ex2}
\centering
\resizebox{120mm}{45mm}{\includegraphics{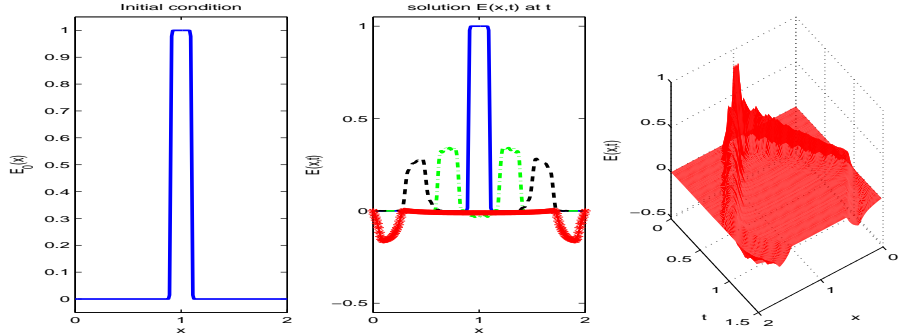}}
\caption{Left: Initial profile. Middle: Evolution of $E(x,t)$ at time points $t=0$ (blue), $t=0.375$ (green), $t=0.75$ (black) and $t=1.275$ (red).  Right:  3D solution illustration of $E(x,t)$. }
\end{figure}

\begin{example}
We consider the two-dimensional Cole-Cole model \eqref{Prob:CC} for $(x,y)\in [0,2]^2$ with smooth initial pulse $u(x,y,0)=\sin(2\pi x)\sin(\pi y/2)$. 
\end{example}

In Figure 3.7, we depict the numerical solutions at different time,   and record the evolution of numerical energy. 
Observe that the numerical solutions at different time  have very similar shapes, but  the magnitude  looks  decreasing as time increases. 
Although the numerical energy does not decay monotonically,  it is bounded by the initial energy (cf.  Theorem \ref{thm:1D} and Remark  \ref{thm:1D}).    



\begin{figure}[!htb] \label{Fig:Ex3}
\centering
\resizebox{40mm}{40mm}{\includegraphics{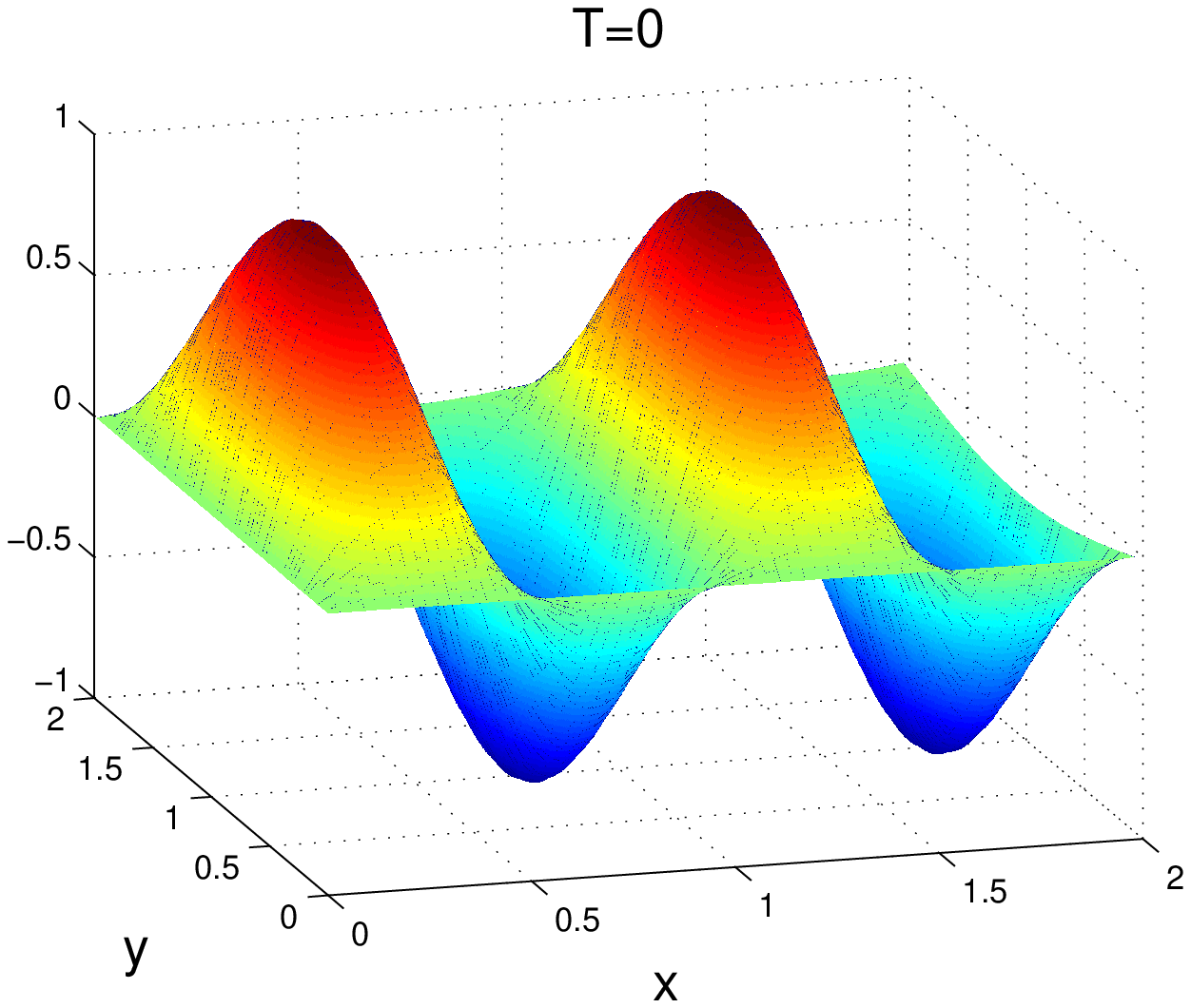}}
\resizebox{40mm}{40mm}{\includegraphics{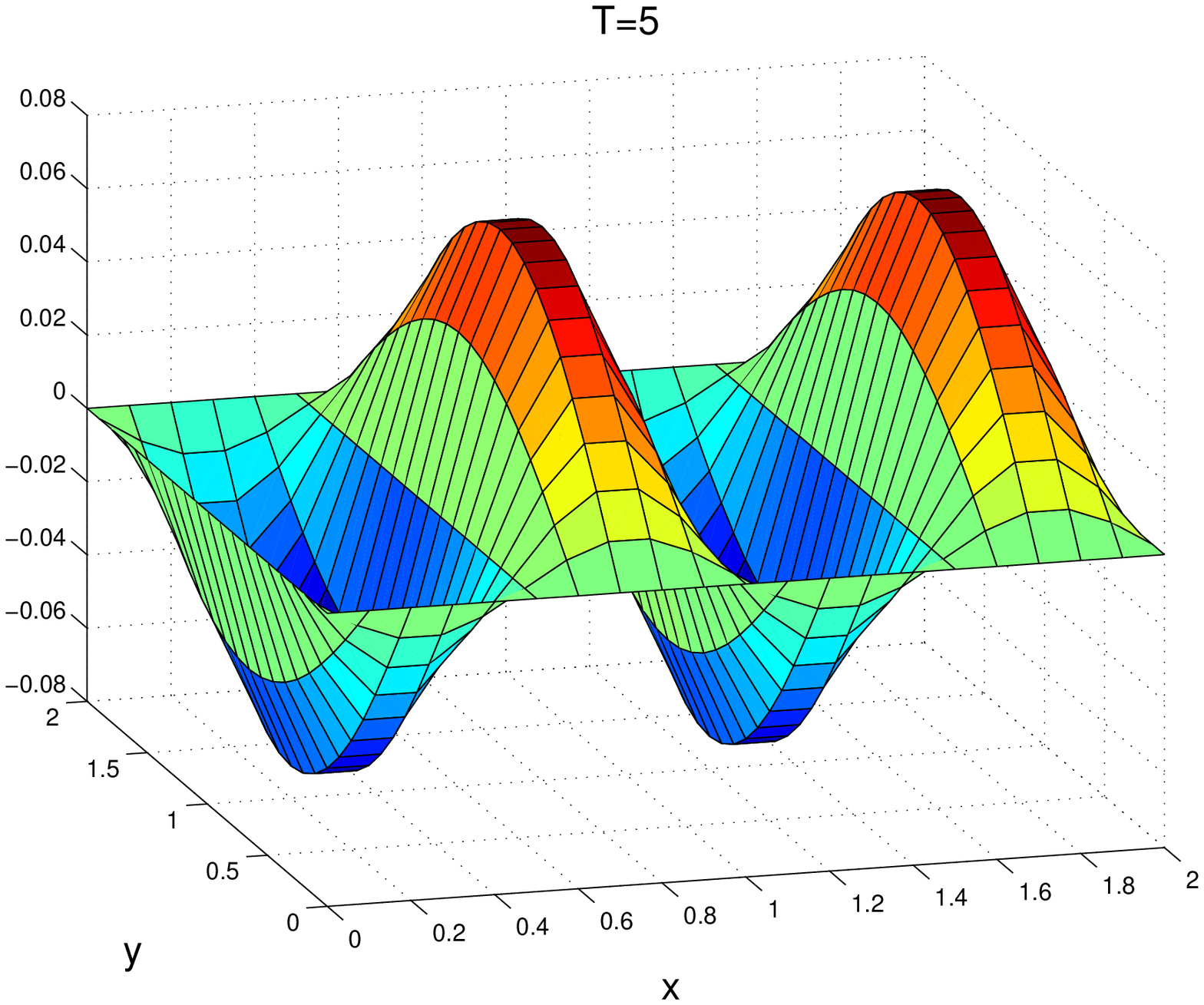}}
\resizebox{40mm}{40mm}{\includegraphics{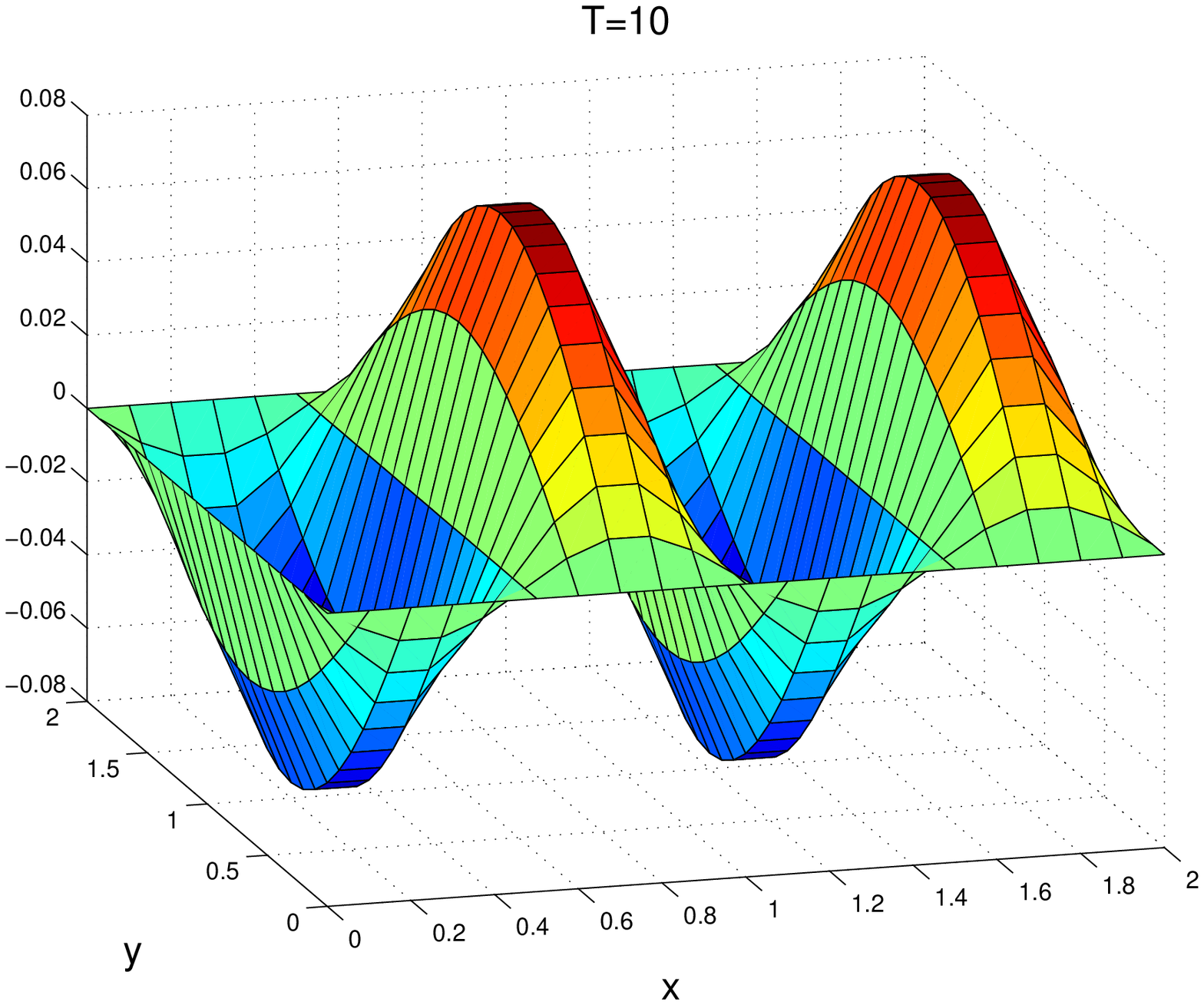}}
\resizebox{40mm}{40mm}{\includegraphics{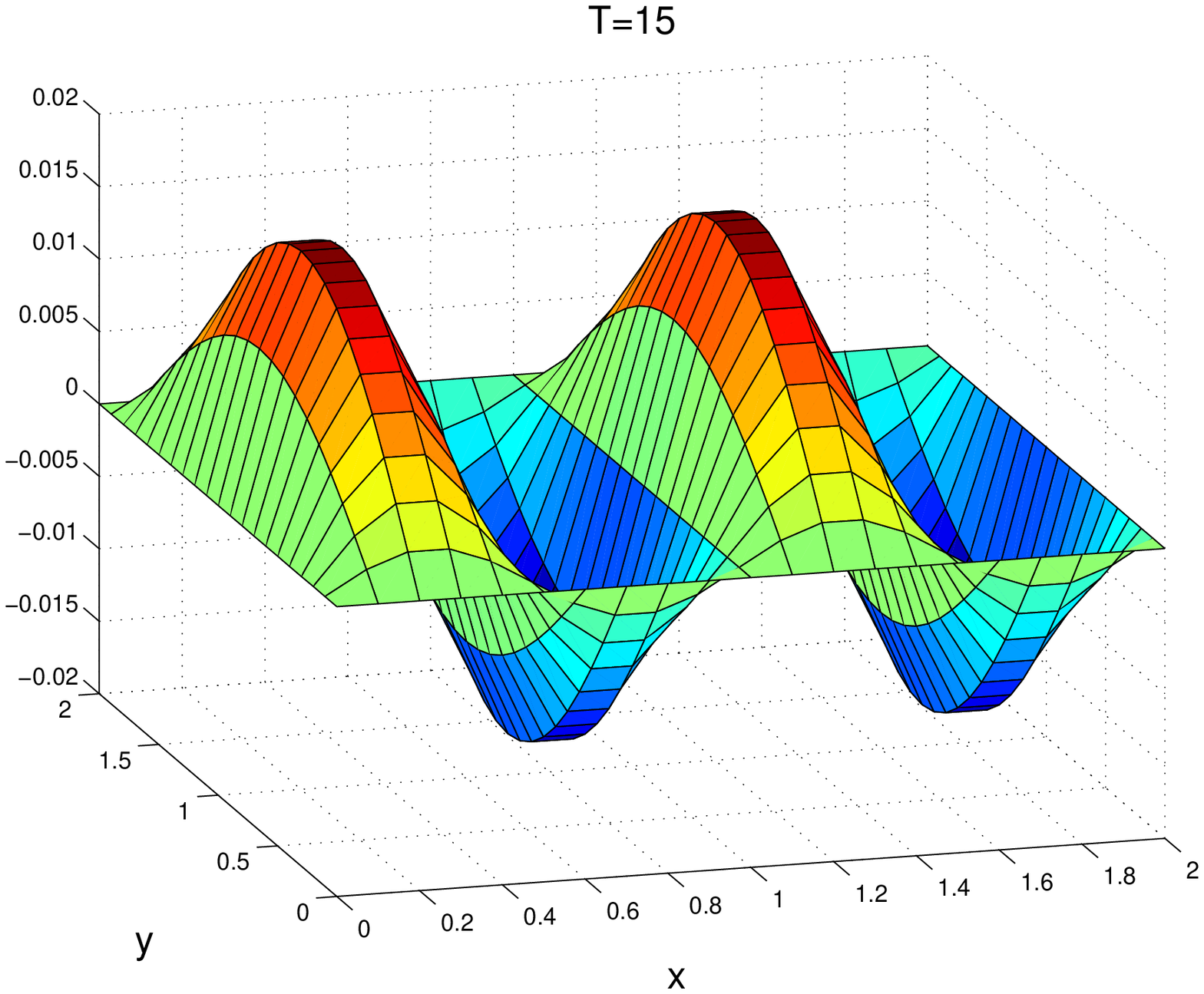}}
\resizebox{40mm}{40mm}{\includegraphics{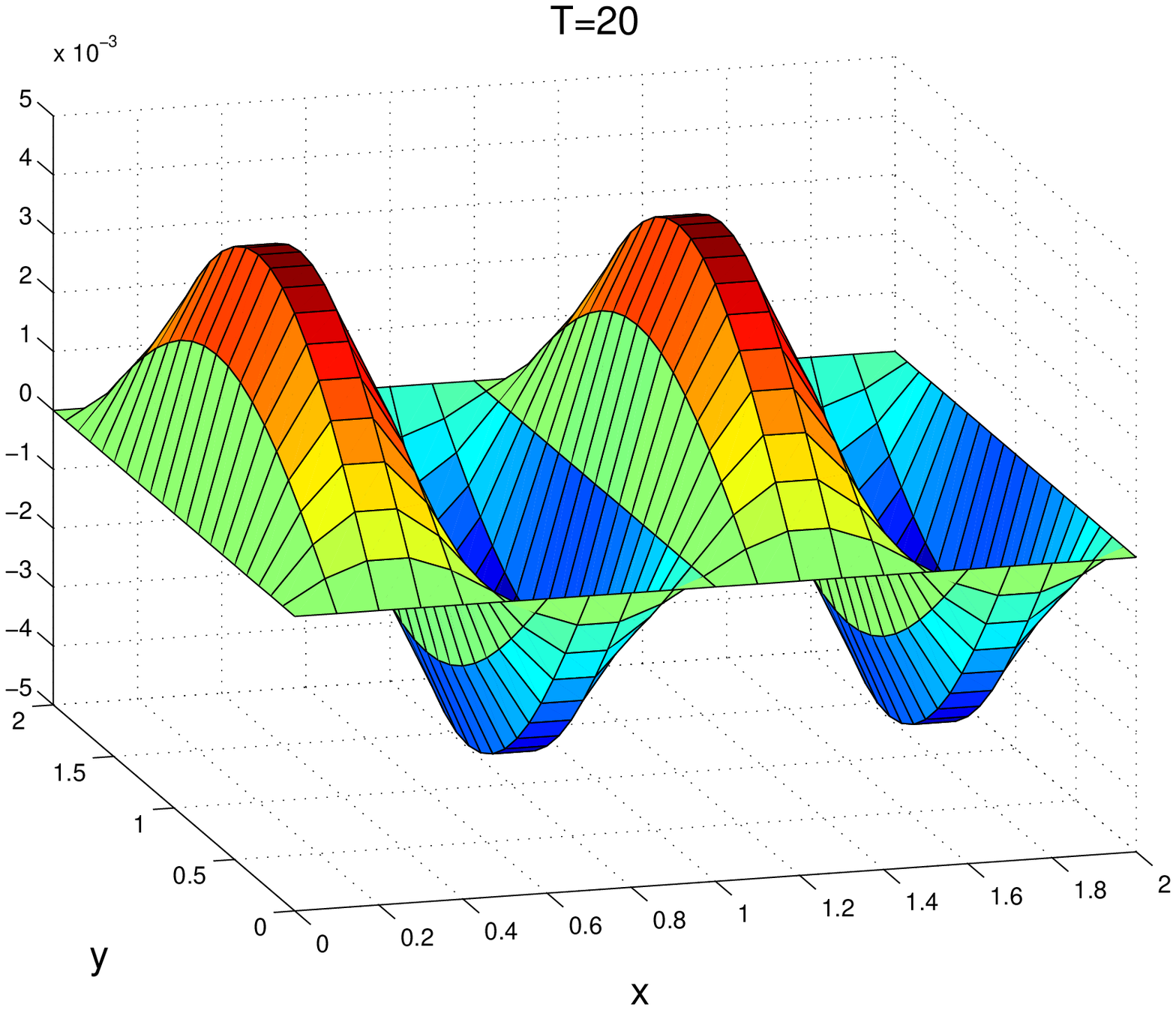}}
\resizebox{40mm}{40mm}{\includegraphics{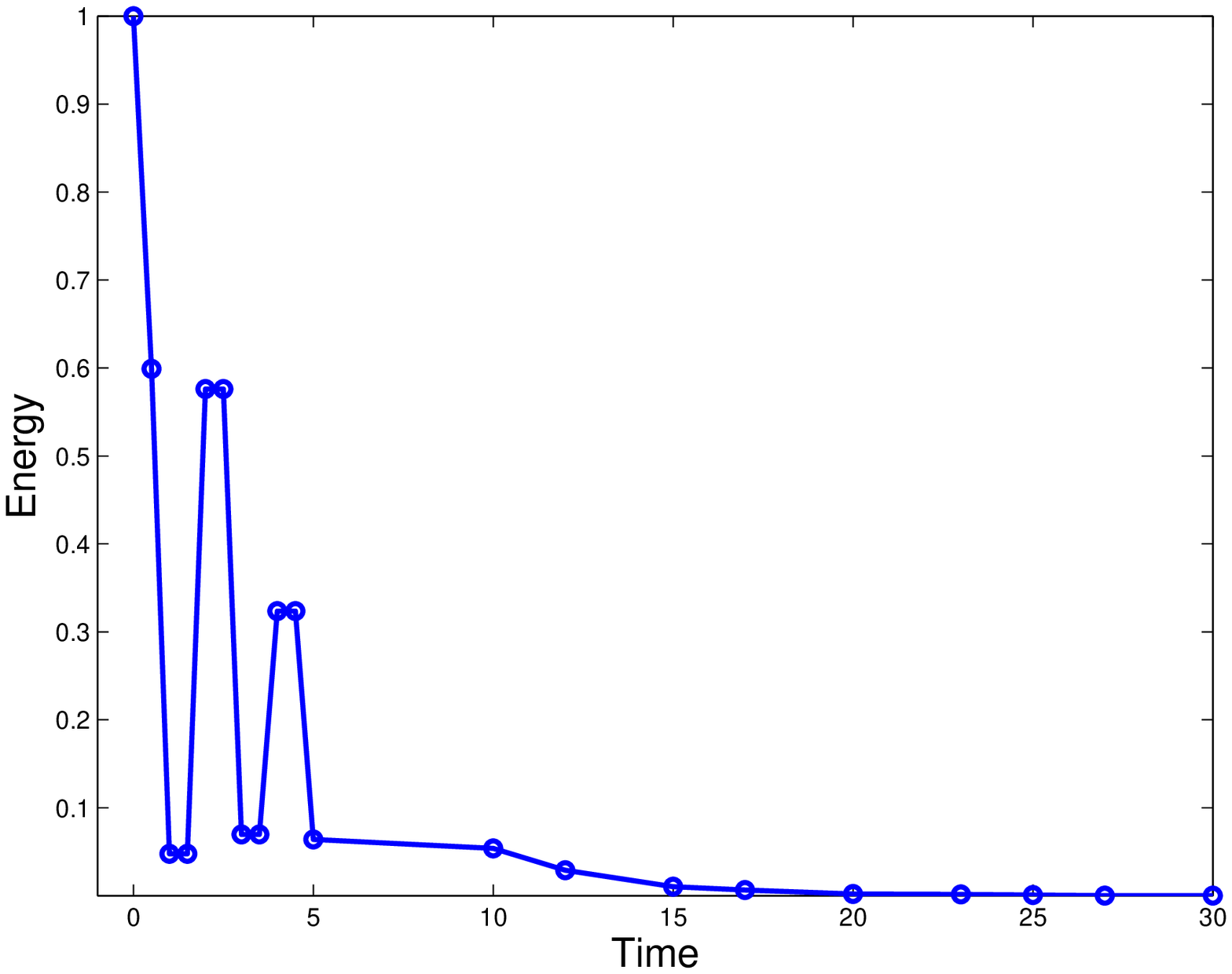}}
\caption{Numerical solution of  \eqref{Prob:CC} with $20$ collocation points on each time interval for time $T=0,5,10,15, 20$, and numerical energy evolution 
with respect to time (the last figure).}
\end{figure}


\section{Discussion and conclusion}\label{Sect4}
In this paper, we have shown that  the high-dimensional Cole-Cole model can be transformed into a temporal PIDE with weakly singular kernel through an adoption of a new variable and electric polarization.  Furthermore, by taking advantage of the special feature of the PIDE, we apply a domain separation technique to convert the equation into a set of ordinary integro-differential equations, and thus greatly reduce the computation cost of the original model.  Moreover, we have carefully exploited the singular behavior of solution of a typical ordinary integro-differential equation and designed a catered numerical algorithm for it. It is noteworthy that
to combat the singular integral in our algorithm, some technical mapped Gauss-Jacobi numerical quadrature seems indispensable.

Another aspect of our algorithm that needs investigation is its fast algorithm counterpart. Similar as the fast algorithm for weakly singular kernel integration \cite{LiR10} or Caputo fractional derivative \cite{JiangZZZ17}, a promising way is applying fast multipole method to find an accurate approximation for the Laplace transform of the kernel $e_{\alpha,\alpha}(-\lambda t^\alpha), 0<\alpha<1, \lambda>0$, or the function $1/(\lambda+s^\alpha)$. Runge's approximation theorem (cf. \cite[p. 61]{SteinS02}) assures the existence of such approximation. This will be our next research topic.

\end{document}